\documentclass[a4paper, 11pt]{article}
\usepackage{graphicx}
\usepackage{amssymb}
\usepackage{amsmath}
\usepackage{amsthm}
\usepackage{color}

\usepackage[top=2 cm,bottom=2 cm,left=2. cm, right=2. cm]{geometry}

\newcommand{\ve}{\varepsilon}

\newtheorem{theorem}{Theorem}[section]
\newtheorem{lemma}[theorem]{Lemma}

\newtheorem{definition}[theorem]{Definition}

\newtheorem{remark}[theorem]{Remark}

\newtheorem{assumption}{Assumption}

\def\XXint#1#2#3{{\setbox0=\hbox{$#1{#2#3}{\int}$ }
\vcenter{\hbox{$#2#3$ }}\kern-.58\wd0}}

\def\T{\mathcal{T}}
\def\e{{\bf e}}
\def\be{{\bf e}}

\definecolor{darkred}{rgb}{0.9,0.1,0.1}

\title{Homogenization of  biomechanical models of plant tissues with randomly distributed cells}
\author{Andrey Piatnitski\footnote{ The Arctic University of Norway, UiT, Campus Narvik, Norway and Institue for
Information Transmission Problems RAS, Moscow, Russia, apiatnitski@gmail.com}   \;  and Mariya Ptashnyk\footnote{ 
Heriot-Watt University, Edinburgh, United Kingdom, m.ptashnyk@hw.ac.uk}
 }

\begin{document}

\maketitle

\begin{abstract}

In this paper homogenization    of a mathematical model for  biomechanics of a plant tissue with randomly distributed cells is considered.  Mechanical properties of a plant tissue are  modelled by   a strongly coupled system of  reaction-diffusion-convection  equations for chemical processes in plant cells and cell walls,   the equations of  poroelasticity for elastic deformations of  plant cell walls and middle lamella,  and the Stokes equations for fluid flow inside the cells.
The nonlinear coupling between the mechanics and chemistry is given by the  dependence of elastic properties of plant tissue  on  densities  of chemical substances    as well as by the dependence  of chemical reactions on mechanical stresses present in a tissue.
 Using  techniques of stochastic homogenization we derive rigorously  macroscopic  model for  plant tissue biomechanics with random distribution of cells.
 Strong stochastic two-scale convergence is shown to  pass to the limit in the non-linear reaction terms.
 Appropriate meaning of the boundary terms is introduced to define the macroscopic equations with flux boundary conditions and transmission conditions  on the microscopic scale.
\end{abstract}



\section{Introduction}

 Formation of plant tissues and organs is a result of the coordinated expansion of hundreds of thousands of cells, different in size, shape, and composition. Plant organs are composed of several types of tissues, e.g.~epidermis, cortex, endodermis, vascular tissue \cite{Nobel}.
While the turgor pressure,  the main force for cell expansion, acts isotropically,  the anisotropic deformation and growth of plant cells
and tissues rely on the  mechanics of cell walls, surrounding plant cells,  and  the microstructure of  cell walls and tissues.
Plant tissues have complex   hierarchical  microstructures given by the size and arrangement of cells, connected by  cross-linked  pectin network of middle lamella,  on one scale, and by the heterogeneous structure of cell walls on  the other scale \cite{Jensen_2015}. In some  tissues, such as wood or cork,  the geometric arrangement of cells is very regular and can be regarded as periodic \cite{Gibson_1999},  however many plant tissues exhibit  random  variations in their microstructure \cite{Faisal_2012, Mathers_2018, Mebatsion_2006}.  Plant cell walls mainly consist of cellulose  microfibrils, pectin, hemicellulose,  macromolecules, and water.
   The orientation of microfibrils, their length, high tensile strength and interaction with wall matrix macromolecules determine  the wall stiffness. For irreversible deformation, the deposition of new wall materials and the loosening of the cell wall through the breaking of the load-bearing cross-links between microfibrils, pectin and hemicellulose by enzymes activity are required \cite{Schopfer}. 
It is supposed that calcium-pectin cross-linking chemistry  strongly influences elastic properties of plant cell walls  \cite{WHH}.  Pectin is produced in Golgi apparatus  inside the cells and is  deposited to  a cell wall in  a  methyl-esterified form, where it  can be de-methylesterified   by the enzyme pectin methylesterase (PME), which removes  methyl groups by  breaking ester bonds. The de-methyl-esterified pectin  is able to form  calcium-pectin cross-links, and so stiffen the cell wall and reduce its expansion, see e.g.~\cite{WG}, whereas   mechanical stresses can  break calcium-pectin cross-links and hence increase the extensibility  of plant cell walls and middle lamella.

Considering the complex structure of plant tissues and organs, for a   better understand  and  improvement of   plant growth and development, it is important to model and  analyse   how  microscopic structure  and  interactions between chemical processes 
and mechanical properties of individual cell walls and cells   contribute to the properties of the plant tissues and organs \cite{Baskin_2013, Jensen_2015}.
Different approaches were applied to analyse the interplay between micro- and macro-mechanics and transport processes  in plant tissues \cite{Band_2012}. Many results can be found for multiscale modelling and analysis of the periodic microstructure of wood \cite{Davalos_2005, Malek_2017, Qing_2010}.  Multiscale  modelling  and analysis of the impact of the microscopic structure of  plant cell walls, especially  orientation and distribution of   microfibrils, on mechanical properties of  cell walls  were conducted in \cite{Ptashnyk_Seguin}.
A vertex-element model and hybrid vertex-midline  model for plant tissue deformation and growth,  coupled with the cell-scale transport of plant hormone,  were considered in \cite{Fozard, Fozard_2016}. The impact of microfibrils on the mechanical properties of cell walls was accounted for by introducing an anisotropic viscous stress which depends on a pair of microfibril directions. 
A simple constitutive model at the cell scale which characterises cell walls via yield and extensibility parameters together with an appropriate averaging over a cross-section were used to  derive the analogous tissue-level model describing  elongation and bending of a plant root  \cite{Dyson_2014}.
A mesh-free particle method was proposed in \cite{Liedekerke_2010}  to simulate the mechanics of both individual plant cells and cell aggregates in response to external stresses  and     to study how plant tissue mechanics is related to the micromechanics of cells.
The interior of the cell is regarded as liquid phase and simulated  using the smoothed particle hydrodynamics (SPH) method,
where in the domain corresponding to the viscoelastic material of cell walls  the particles are connected by pairwise interactions holding them together.
A multiscale method for the simulation of large viscoelastic deformations of a plant tissue presented in \cite{Ghysels_2010} combines particle method on the microscopic level with standard finite elements methods on the macroscopic scale.
The effect of non-periodic microstructure  on effective (homogenized)  elastic properties of  two-dimensional cellular materials (honeycombs) was studied in \cite{Silva_1995}  by considering non-periodic arrangement of cell walls in random Voronoi honeycombs and applying  finite element analysis. The finite-edge centroidal Voronoi tessellation (FECVT)  was introduced in  \cite{Faisal_2012} to generate a realistic model of a non-periodic tissue microstructure and, combined with finite elements analysis, was used to determine the effective elastic properties of plant tissues, especially plant petioles and stems \cite{Faisal_2013}. 
Smoothed particles hydrodynamics (SPH) framework was used in \cite{Mimault_2019} to model plant tissue growth. The framework identifies the SPH particle with individual cells in a tissue, but the tissue growth is performed at the macroscopic level using SPH approximations and plant tissue is represented as an anisotropic poro-elastic material. 
A coarse-grained multiscale numerical model is proposed in \cite{Wijerathne_2019} to predict  macroscale  deformations of food-plant tissues (e.g.~apple tissues) during drying.

In \cite{Andrey_Mariya} we derived and analysed a mathematical model for plant tissue biomechanics, which describes  the interactions between  calcium-pectin dynamics and  deformations  of a plant tissue.  The microscopic model, at  the  length scale of plant cells, comprises     a strongly coupled system of the Stokes equations modelling water flow inside  plant cells,  the equations of poro-elasticity defining elastic deformations of   plant cell walls  and middle lamella, and reaction-diffusion-convection equations describing the dynamics of the  methyl-esterfied pectin, de-methyl-esterfied pectin, calcium ions, and calcium-pectin cross-links. The interplay between the mechanics and the chemistry   comes in by assuming that the elastic  properties of  cell walls and middle lamella depend on the density of the calcium-pectin cross-links and the  stress within  cell walls and middle lamella   can break the cross-links.   Assuming periodic distribution of cells in a plant tissue in \cite{Andrey_Mariya} we derived rigorously macroscopic  model   for plant tissue biomechanics.  The two-way coupling between chemical processes and mechanics is the main novelty of the model,  which also induces some non-standard elements in the analysis of the model and in the rigorous  derivation of macroscopic equations.   In this paper we generalise the results obtain in \cite{Andrey_Mariya}  by  considering  random distribution of cells in  a plant tissue,   observed experimentally in many plant tissues and organs \cite{Faisal_2012,  Mebatsion_2006}.  The derivation  of macroscopic equations  from a continuum  description of the microscopic  processes on the cell  level using stochastic homogenization techniques results into a continuum macroscopic two-scale model   containing the information on the microscopic interactions.  Our microscopic model incorporates microscopic properties of plant cell walls, essential for plant tissue mechanics.  The macroscopic model takes into account the microscopic structure of a plant tissue via effective (macroscopic) elasticity and permeability tensors 
and includes   the interplay between the fluid in cell inside and poroelastic nature  of cell walls and middle lamella.   The effect of the microstructure and heterogeneity of the processes is also reflected in the equations for calcium-pectin chemistry  via effective (macroscopic) diffusion coefficients, reaction terms and advective velocity. In the relation to particle and vertex-elements methods, continuum modelling approach proposed here  may be beneficial when consider large size plant tissues and organs. 

  To analyse macroscopic   mechanical properties  of  plant  tissues	with a random distribution of cells,     we derive rigorously a  macroscopic model for plant biomechanics  using  techniques of stochastic homogenization.
The stochastic two-scale convergence \cite{Zhikov_Piatnitski_2006} is applied to obtain the macroscopic equations.    The main mathematical difficulties in the derivation of the macroscopic equations  arise from the strong coupling between the equations of poro-elasticity    and the system of reaction-diffusion-convection equations, as well as    due to transmission conditions between the free fluid and poro-elastic material. The strong stochastic two-scale convergence for the displacement gradient and flow velocity  is proven to pass to the  limit in the nonlinear reactions terms.
Extension arguments and formulations of surface integrals as volume integrals are used to pass to the stochastic two-scale limit in the equations with non-homogeneous Neumann boundary conditions and transmission conditions. To pass to the limit in the flux boundary conditions defined on the surfaces of the microstructure, Palm measure and the proven here  trace inequality for $H^1$-function in the  probability space, see Lemma~\ref{converg_bound_Gamma}, are used.

 Some of the first  results on the stochastic homogenization of linear second-order elliptic equations  were obtained in  \cite{Kozlov1980, Papanicolaou1979, ZKON1979}. The homogenization of quasi-linear elliptic  and parabolic equations with stochastic coefficients  and convex integral operators was considered  in \cite{Bensoussan1988, Castell2001,DalMaso1986_1, DalMaso1986_2}.  Subadditive ergodic theory and the method of  viscosity solutions were applied   to homogenize  Hamilton-Jacobi,  viscous Hamilton-Jacobi equations, and fully nonlinear  elliptic and parabolic equations  in stationary ergodic media \cite{Armstrong2012, Caffarelli2005, Kosygina2006, Souganidis2005, Souganidis2010} (see also references therein).
 The stochastic two-scale convergence introduced in \cite{Zhikov_Piatnitski_2006} has been  extended to Riemannian manifolds and has been applied to analyze heat transfer through composite and polycrystalline  materials with nonlinear conductivities \cite{Heida_2012, Heida_2011}.
  The  two-scale convergence in the mean \cite{BMW1994}  has been applied to derive macroscopic equations for  single- and two-phase fluid flows in randomly fissured media  \cite{BMP2003, Wright2001}.

The poro-elastic equations, modelling   interactions   between  fluid flow  and elastic deformations of  a porous medium, has been first  obtained  by Biot  using  a phenomenological approach \cite{Biot_1972, Biot_1962, Biot_1941} and subsequently  derived by applying formal asymptotic expansion \cite{Auriault, Burridge, Levy, SP} or  the two-scale convergence method   \cite{Clopeau, GM, JMN,  Meirmanov_2008, Meirmanov_2007, Nguetseng_1990}.
Along many results for poroelastic equations, only few studies  of  interactions between a free fluid and a deformable porous medium can be found. 
 In \cite{Showalter_2005} nonlinear semigroup method was used for mathematical analysis of a system of poroelastic equations coupled with the Stokes equations for free fluid flow.  A rigorous derivation of  interface conditions between a poroelastic medium  and an elastic body  was considered in \cite{MW}.  Numerical methods  for  coupled system of poroelastic   and Navier-Stokes equations were studied in~\cite{Badia, Bukac}.

 One of the approaches commonly used in numerical homogenization to approximate the effective coefficients
 of a microscopic problem describing some processes in a random medium is the so-called periodization \cite{BP2004}. The key idea 
 of this method is to choose a large enough sample of the random medium,  to extend it periodically, and to take the 
 effective coefficients of the obtained periodic problem as an approximation of the effective coefficients of the original
random problem.    Recent years an essential progress was achieved in this approach, see  the work \cite{Gloria_2015}, 
and references therein. Justification of this method for the model studied in the present paper
is an interesting problem.  Mixed multiscale finite element  method  \cite{Aarnes_2008} or   stochastic variational multiscale  method  \cite{Asokan_2006}     can also be used for numerical simulation of multiscale stochastic problems. 

The paper is organised as follows. In Section~\ref{model} we  formulate the microscopic model for plant tissue biomechanics. The main results of the paper are summarised in Section~\ref{main_results}.
The a priori estimates and convergence results are given in Sections~\ref{apriori_1} and \ref{convergence}.
In Section~\ref{macro_elasticity} we derive macroscopic equations for  the coupled poro-elastic  and Stokes problem.
The strong stochastic two-scale convergence for displacement  gradient and flow velocity is proven in Section~\ref{strong_convergence}.
  The macroscopic equations for the system of reaction-diffusion-convection equations are derived  in Section~\ref{macro_diffusion}.

\section{Microscopic model}\label{model}

We consider a probability space $(\Omega, \mathcal F, \mathcal P)$ with probability measure $\mathcal P$.
We define  a $3$-dimensional dynamical system $\T_x:\Omega \to \Omega$, i.e.\  a family $\{\T_x\, : \, x\in\mathbb{R}^3\}$ of invertible maps, such that for each $x\in \mathbb R^3$,  $\T_x$ is  measurable and satisfy the following conditions:
\begin{enumerate}
\item[(i)] $\T_0$  is the identity map on  $ \Omega$, and for all  $ x_1, \, x_2 \in \mathbb R^3$    the semigroup property holds:
$$ \T_{x_1+x_2}= \T_{x_1}\T_{x_2}. $$
\item[(ii)]  $\mathcal P$ is an invariant measure for $\T_x$, i.e.\ for each $x\in \mathbb R^3$ and $F\in \mathcal F$ we have that $$\mathcal P(\T^{-1}_x F) = \mathcal P(F).$$
\item[(iii)] For each $F\in \mathcal F$,  the set $\{ (x,\omega)\in \mathbb R^3\times \Omega : \T_x \omega \in F\} $ is a $\mathcal{L}\times \mathcal{F}$-measurable subset of $\mathbb R^3 \times \Omega$, where $\mathcal{L}$ denotes the Lebesgue $\sigma$-algebra on $\mathbb R^3$.
\end{enumerate}

 We consider a fixed measurable set   $\Omega_f$ such that  $\mathcal P(\Omega_f)>0$ and $\mathcal P(\Omega\setminus \Omega_f) >0$ and  denote $\Omega_e = \Omega \setminus \Omega_f$.  We also consider $\Omega_\Gamma \subset \Omega$,  with $\mathcal P(\Omega_\Gamma)>0$ and $\mathcal P(\Omega_\Gamma\cap \Omega_j)>0$,  for $j=e,f$.

 For $\mathcal P$-a.a.\ $\omega\in \Omega$ we define the following  random subdomains  in $\mathbb R^3$
 $$
 G_j(\omega) = \{ x\in \mathbb R^3: \; \T_x  \omega \in \Omega_j\}, \; \; \; \text{ for } j=e,f,\;
 \; \quad G_\Gamma(\omega) = \{ x\in \mathbb R^3: \; \T_x  \omega \in \Omega_\Gamma\},
 $$
 and surfaces
  $$ \Gamma(\omega) = \partial G_f(\omega), \;\qquad \widetilde \Gamma(\omega) = \Gamma(\omega)  \cap G_\Gamma(\omega).
  $$
 We shall consider the following  assumptions on $G_f(\omega)$,  $\Gamma(\omega)$, and  $\widetilde\Gamma(\omega)$:
 \begin{assumption}\label{assum1}
\quad  \phantom{22}
 \begin{itemize}
 \item[1.]  $G_f(\omega)$ consists of countable number of disjoined Lipschitz domains for  $\mathcal P$-a.a.\  $\omega \in \Omega$ with  a uniform    Lipschitz constant.
\item[2.] The distance between two connected components of $G_f(\omega)$ is uniformly bounded from above and below.
\item[3.]  The diameter of  connected components of  $G_f(\omega)$ is bounded from below and above by some positive constants.
\item[4.] The surface  $\widetilde \Gamma(\omega) \subset \Gamma(\omega)$  is  open on $\Gamma(\omega)$ and  Lipschitz continuous.
\end{itemize}
\end{assumption}

Consider a bounded $C^{1, \alpha}$-domain $G\subset \mathbb R^3$, with $\alpha >0$,  representing a part of a plant tissue.  In a plant tissue individual cells,  consisting  of cell inside and cell walls, are connected by the pectin network of middle lamella.  Then the microscopic structure of a plant tissue with a random distribution of  cells is  defined
as  $$
 G_f^\ve = \{ x\in \mathbb R^3: \; \T_{x/\ve} \omega \in \Omega_f \} \cap G, \quad  G_\Gamma^\ve = \{ x\in \mathbb R^3: \; \T_{x/\ve} \omega \in \Omega_\Gamma \} \cap G,  \quad  G_e^\ve = G \setminus G_f^\ve,
 $$
 $$ \Gamma^\ve = \partial G_f^\ve,  \qquad \widetilde \Gamma^\ve=  \Gamma^\ve \cap G^\ve_\Gamma,
$$
$\mathcal P$-a.s., where $G_e^\ve$ represent the subdomains occupied by cell walls and middle lamella, $G_f^\ve$ denotes the  cell inside, and $\widetilde \Gamma^\ve$ defines a part of cell membrane which is impermeable to calcium ions.

Assumption~\ref{assum1}$.2$  states that  the thickness of cell walls and middle lamella is uniformly bounded from above and below and
Assumption~\ref{assum1}$.3$ postulates that the diameter of cells is bounded from above and below.

Due to  assumed random distribution of cells in a plant tissue,   the permeability  and  elastic properties of plant cell walls and middle lamella  are characterised by the corresponding random variables. 
For this  we define statistically homogeneous random fields
${\bf E}_1(x,\omega, \xi) = {\bf \widetilde E}_1(\T_x\omega, \xi) $ and $K_p(x,\omega) = \widetilde K_p(\T_x\omega)$, where
${\bf \widetilde E}_1(\cdot, \xi)$ and $\widetilde K_p(\cdot)$ are given measurable functions from $\Omega$ to  $\mathbb R^{3^4}$ and $\mathbb R^{3\times 3}$, respectively, for $\xi \in \mathbb R$   representing the dependence of the elastic properties on the calcium-pectin cross-links density.  It is observed experimentally that the load bearing calcium-pectin cross-links reduce cell wall expansion, see e.g.~\cite{WHH}, and hence we assume that elastic properties of cell walls and middle lamella depend on the  density of calcium-pectin cross-links.

Then for each $\omega \in \Omega$ and $\xi \in \mathbb R$ the microscopic elasticity tensor  ${\bf E}_1^\ve$ and permeability tensor $K_p^\ve$ are defined as
\begin{equation}\label{EK_def}
{\bf E}_1^\ve(x, \xi) = {\bf E}_1(x/\ve, \omega, \xi), \quad K_p^\ve(x)= K_p(x/\ve, \omega).
\end{equation}

In the mathematical model for biomechanics of a plant tissue we consider concentration of  calcium  $c_e^\ve$ and $c_f^\ve$  in   cell walls  and middle lamella  $G_e^\ve$ and in cell cytoplasm $G_f^\ve$ (cell inside), respectively.  In addition,  in  the domain of  cell walls  and middle lamella  $G_e^\ve$  densities of  methylesterified and de-methylesterified   pectins $n^\ve_e$ and  $n^\ve_d$ and  of calcium-pectin cross-links  $n^\ve_b$   are considered.     We shall use the notation $b^\ve_e= (n_e^\ve, n_d^\ve, n_b^\ve) $
and $D_b(b_{e,3}^\ve)= \text{diag} (D_{n_e}(n_b^\ve), D_{n_d}(n_b^\ve), D_{n_b}(n_b^\ve))$ denotes  the diagonal matrix of diffusion coefficients for $n_e^\ve$,  $n^\ve_d$, and   $n^\ve_b$ respectively.
We assume that the inflow of new calcium is facilitated only on parts of the cell membrane $\Gamma^\ve \setminus \widetilde \Gamma^\ve$.   Here we consider   a passive flow of calcium between cell wall and cell inside.  The regulatory mechanism for calcium inflow by mechanical properties of cell walls will be considered in further studies.
For elastic deformations of plant cell walls and middle lamella  we consider homogenized equations of poro-elasticity reflecting the microscopic structure of cell walls composed of elastic cellulose microfibrils and cell wall matrix permeable for the fluid flow.
The differences in the elastic properties of cell walls and middle lamella are reflected in the elasticity tensor ${\bf E}_1^\ve$, which depends on the microscopic variable $x/\ve$.
Here we consider  diffusion coefficients depending on calcium-pectin cross-links density.  The analysis in the case of diffusion coefficients depending additionally on microscopic and macroscopic variables will follow the same steps.

We shall use the notations $G_T = (0,T)\times G$, $(\partial G)_T = (0,T)\times \partial G$,  $G_{j,T}^\ve = (0,T)\times G_j^\ve$ for $j=e,f$,   $\Gamma_T^\ve = (0, T) \times \Gamma^\ve$, and $\widetilde \Gamma_T^\ve = (0, T) \times \widetilde \Gamma^\ve$.  By $\Pi_\tau w$ we define the tangential projection of a vector $w$, i.e.\ $\Pi_\tau w= w - (w\cdot n) n$, where $n$ is a normal vector and $\tau$ indicates the tangential subspace to the boundary.

For $\mathcal P$-a.a.\ realisations $\omega \in \Omega$ the microscopic model for the concentration of  calcium   and densities of  pectins and calcium-pectin cross-links reads
\begin{equation}\label{eq_codif}
\begin{aligned}
&\partial_t b_e^\ve = {\rm div}(D_b(b^\ve_{e,3})\nabla b^\ve_e) + g_b(c_e^\ve, b_e^\ve, \e(u_e^\ve) ) , \qquad &&  \text{ in } G_{e, T}^\ve,\\
&\partial_t c_e^\ve = {\rm div}(D_e(b^\ve_{e,3})\nabla c^\ve_e) + g_e ( c^\ve_e, b_e^\ve, \e(u_e^\ve))  \qquad &&  \text{ in }  G_{e, T}^\ve ,\\
&\partial_t c^\ve_f = {\rm div}(D_f\nabla c_f^\ve - \mathcal{G}(\partial_t u_f^\ve) c_f^\ve)  + g_f(c_f^\ve)\qquad &&  \text{ in } G_{f, T}^\ve,\\
&D_b(b^\ve_{e,3})\nabla b_e^\ve \cdot n =\ve \, R(b_e^\ve) && \text{ on }   \Gamma^\ve_T,  \\
&c_f^\ve = c_e^\ve  \qquad  && \text{ on }\Gamma^\ve_T\setminus\widetilde \Gamma^\ve_T, \\
&D _e(b^\ve_{e,3})\nabla c_e^\ve \cdot n = (D_f\nabla c_f^\ve -   \mathcal G(\partial_t u_f^\ve) c_f^\ve)\cdot n  \qquad  && \text{ on } \Gamma^\ve_T\setminus\widetilde \Gamma^\ve_T, \\
&D_e(b^\ve_{e,3})\nabla c_e^\ve \cdot n =0, \qquad (D_f \nabla c_f^\ve- \mathcal G(\partial_t u_f^\ve) c_f^\ve) \cdot n = 0 \quad && \text{ on } \widetilde \Gamma^\ve_T.
\end{aligned}
\end{equation}
Here $u^\ve_e$ stands for the displacement from the equilibrium position in poroelastic material of cell wall and
middle lamella, $\mathbf{e}(u^\ve_e) = (\mathbf{e}(u^\ve_e)_{ij})_{i,j=1,2,3}$ for  its symmetrized gradient, with $\mathbf{e}(u^\ve_e)_{ij} = (\partial_{x_i} u^\ve_{e j} + \partial_{x_j} u^\ve_{e i})/2$, and   $\partial_t u_f^\ve$ denotes
the fluid velocity in the cell inside.  The pressures in the poroelastic and fluid domains are denoted by $p_e^\ve$
and $p_f^\ve$, respectively.   The function $\mathcal{G}$ defines the velocity field in the convection term  in cell inside and  is a Lipschitz continuous bounded
function of the intracellular flow velocity $\partial_t u_f^\ve$. The condition that $\mathcal{G}$ is bounded is natural from the biological and physical point of view, because the flow velocity in plant tissues is bounded. This condition is also essential for the derivation of \textit{a priori} estimates.

The  water   flow inside the cells and elastic deformations of plant cell walls  and middle lamella   are modelled by  a coupled system of  poro-elastic  and Stokes equations
\begin{equation}\label{equa_cla}
\begin{aligned}
&\rho_e \partial^2_t u^\ve_e- {\rm div} ( {\bf E}^\ve(b^\ve_{e,3}) \e( u^\ve_e)) + \nabla p_e^\ve  = 0  \qquad &&  \text{ in } G_{e,T}^\ve, \\
&\rho_p\partial_t p^\ve_e- {\rm div} ( K^\ve_p \nabla p^\ve_e -  \partial_t u_e^\ve)  = 0  \qquad &&  \text{ in }  G_{e,T}^\ve, \\
&\rho_f \partial^2_t u_f^\ve -\ve^2 \mu \, \text{div} ( \e(\partial_t u^\ve_f)) + \nabla p_f^\ve = 0&&  \text{ in } G_{f,T}^\ve, \\
&\text{div } \partial_t u_f^\ve = 0 &&  \text{ in } G_{f,T}^\ve , \\
& ({\bf E}^\ve(b^\ve_{e,3}) \, \e( u^\ve_e) - p_e^\ve I) n = ( \ve^2 \mu \, \e(\partial_t u^\ve_f)  - p_f^\ve I) n &&  \text{ on } \Gamma^\ve_T, \\
&\Pi_\tau \partial_t u_e^\ve =\Pi_\tau \partial_t u_f^\ve, \qquad  n \cdot (\ve^2 \mu \,  \e(\partial_t u^\ve_f) - p_f^\ve I) n  = - p_e^\ve \qquad
  &&  \text{ on } \Gamma^\ve_T, \\
&(-K^\ve_p \nabla p^\ve_e +  \partial_t u_e^\ve)\cdot n = \partial_t u_f^\ve\cdot n  &&  \text{ on } \Gamma^\ve_T, \\
& u_e^\ve(0,x) = u^\ve_{e0}(x), \quad  \partial_t u_e^\ve (0,x) = u_{e0}^1(x), \quad   p_e^\ve (0,x) = p^\ve_{e0}(x) && \text{ in } G_e^\ve, \\
 & \partial_t u_f^\ve (0,x) =u^1_{f0}(x) \quad   &&  \text{ in } G_f^\ve,
\end{aligned}
\end{equation}
 where $\rho_e$ denotes the poroelastic wall density,  $\rho_p$ is the mass storativity coefficient, and $\rho_f$ denotes the fluid density.  We assume that $\rho_e$, $\rho_p$, and $\rho_f$ are positive and constant. 
The dependence of the elastic properties of the cell wall matrix and middle lamella on calcium-pectin cross-links is reflected in the dependence of the elasticity tensor ${\bf E}^\ve$ on $b^\ve_{e,3}(\cdot)$.
In what follows we assume that this dependence is non-local in temporal variable which reflects the time of reaction, i.e.\ the stretched cross-links have different impact (stress drive hardening) on the elastic properties of the cell wall matrix than newly-created cross-links, see e.g.~\cite{Broedersz, Proseus_2006, Schuster_2012}.
More precisely, we assume  in \eqref{EK_def} that $\xi=\mathcal{K}\big(b^\ve_{e,3}(\cdot)\big)(t ,x) =\int_0^t \kappa (t-\tau) b^\ve_{e,3}(\tau, x) d \tau$, where  $\kappa (\cdot)$ is a smooth
non-negative kernel, and  define
$$
\widetilde{\bf E}\big(\omega, b^\ve_{e,3}(\cdot)\big)=\widetilde{\bf E}_1\Big(\omega,
\int_0^t \kappa (t-\tau) b^\ve_{e,3}(\tau, x) d \tau\Big),\qquad
{\bf E}^\ve\big(x, b^\ve_{e,3}(\cdot)\big) = \widetilde{\bf E}\big(\mathcal T_{x/\ve} \omega, b^\ve_{e,3}(\cdot)\big).
$$
Together with the profile of function ${\bf E}_1^\ve$ this kernel specifies how the elastic properties of  cell walls and middle lamella  depend on calcium-pectin cross-links, see Assumption {\bf A2} for further conditions on $\kappa$.

On the external boundaries we consider some given forces applied to plant tissues and flux conditions for pectins and calcium:
\begin{equation}\label{exbou_co}
\begin{aligned}
&D_b\nabla b_e^\ve \cdot n = F_b(b_e^\ve), \quad  D_e\nabla c_e^\ve \cdot n = F_c(c_e^\ve) && \text{ on } (\partial G)_T, \\
& {\bf E}^\ve(b^\ve_e) \e( u^\ve_e) \, n = F_u  \quad && \text{ on } (\partial G)_T, \\
& (K^\ve_p \nabla p^\ve_e -  \partial_t u_e^\ve) \cdot n = F_p  &&  \text{ on }  (\partial G)_T.
\end{aligned}
\end{equation}
A detailed derivation of the model equations  \eqref{eq_codif} and \eqref{equa_cla} can be found in \cite{Andrey_Mariya}.

System \eqref{eq_codif}--\eqref{exbou_co} is studied under the  following assumption on the  coefficients and nonlinear functions:

\bigskip
\noindent
\begin{assumption}\label{assumptions1}
${ }$
\begin{itemize}
\item [\bf A1.]  $D_b^{jj}, D_e \in C(\mathbb R)$  such that
$d_j\leq D_b^{jj}(\xi) \leq \tilde d_j$ and $d_e\leq D_e(\xi) \leq \tilde d_e$ for all $\xi \in \mathbb R$, with some $d_j, d_e, \tilde d_j, \tilde d_e>0$ and $j=1,2,3$.
\item [\bf A2.]  Elasticity tensor ${\bf \widetilde E}(\omega,\xi) = ( \widetilde E_{ijkl}(\omega,\xi))_{1\leq i,j,k,l\leq 3}$ satisfies  $\widetilde E_{ijkl} = \widetilde E_{klij} =\widetilde  E_{jikl} =\widetilde E_{ijlk}$   and
$\alpha_1 |A|^2  \leq  \widetilde{\bf E}(\omega, \xi) A \cdot A   \leq \alpha_2 |A|^2 $ for  all symmetric  $A \in \mathbb R^{3 \times 3}$,   $\xi \in \mathbb R_+$, $\mathcal P$-a.a.\ $\omega \in \Omega$,  and  $0<\alpha_1 \leq \alpha_2 < \infty$,\\
 ${\widetilde {\bf E}}(\omega, \varsigma(\cdot)) =\widetilde{\bf E}_1(\omega, \mathcal K(\varsigma(\cdot)))$,  where
 $\widetilde{\bf E}_1\in C(\Omega; C^2_b( \mathbb R))$ and  $\mathcal K(\varsigma(\cdot)) = \int_0^t \kappa (t-\tau)
 \varsigma (\tau, x) d \tau$,
with a smooth function $\kappa: \mathbb R_+ \to \mathbb R_+$  such that $\kappa(0)=0$.
\item [\bf A3.]  $\widetilde K_p \in L^\infty(\Omega)$ and $\widetilde K_p(\omega)\eta \cdot \eta \geq k_1 |\eta|^2$ for $\eta \in \mathbb R^3$,  $\mathcal P$-a.a.\  $\omega \in \Omega$,  and $k_1 >0$.
\item [\bf A4.]  The convection function $\mathcal{G}$ is a Lipschitz continuous function on $\mathbb R^3$ such that $|\mathcal{G}(r)|\leq \rho$, for some $\rho>0$.
\item [\bf A5.] For functions $g_b$, $g_e$, $g_f$, $R$, $F_b$, and $F_c$ we assume
$$
g_b\in C(\mathbb R\times\mathbb R^3\times \mathbb R^6;\mathbb R^3), \quad
g_e\in C(\mathbb R\times\mathbb R^3\times \mathbb R^6), \quad
F_b,\,R\in C(\mathbb R^3;\mathbb R^3),
$$
and  $F_c$  and $g_f$  are Lipschitz continuous.   Moreover, the following estimates hold:
$$
\begin{aligned}
&|g_b(s,r, A )|\leq C_1(1+|s|+|r|) + C_2|r||A|,   \quad  && |F_b(r)| + |R(r)| \leq C(1+ |r|),  \\
 &|g_e(s,r, A)|\leq C_1(1+|s|+ |r|)+ C_2 (|s|+|r|)|A|,\quad  && |F_c(s)| + |g_f(s)|\leq C(1+|s|),
\end{aligned}
$$
where $s\in\mathbb R_+$, $r \in \mathbb R^3_+$,  and $A$ is a symmetric $3\times 3$ matrix. Here and in what follows we identify the space of symmetric $3\times3$ matrices with $\mathbb R^6$.

  It is also assumed that for any symmetric $3\times 3$ matrix $A$  we have  that  $g_{b,j}(s,r,A)$, $F_{b,j}(r)$, $R_j(r)$  are non-negative for  $r_j=0$,  $s\geq 0$, and  $r_i\geq 0$, with $i=1,2,3$ and $j\neq i$,   and $g_e(s,r,A)$, $g_f(s)$, and $F_c(s)$  are non-negative for $s=0$ and $r_j\geq 0$,  with $j=1,2,3$.

We assume also  that   $F_b$ and $R$ are locally Lipschitz continuous,   and
$$
\begin{aligned}
&  |g_b(s_1,r_1, A_1 ) - g_b(s_2,r_2, A_2 )| \leq C_1(|r_1|+|r_2| )|s_1-s_2|  + C_2 (|s_1|+ |s_2|+|A_1|+|A_2|)|r_1-r_2| \\ &
 \hspace{ 5  cm }  +  C_3(|r_1|+|r_2| ) |A_1-A_2|, \\
  &  |g_e(s_1,r_1, A_1 ) - g_e(s_2,r_2, A _2)|\leq C_1(|r_1|+|r_2| + |A_1|+ |A_2|)|s_1-s_2|  \\
  & \hspace{ 2.2  cm }  +C_2(|s_1|+ |s_2|+|A_1| + |A_2|)|r_1-r_2| + C_3 (|r_1|+|r_2| + |s_1|+ |s_2|) |A_1-A_2|,
\end{aligned}
$$
for $s_1, s_2\in ( -\mu, + \infty)$, $r_1, r_2 \in  ( - \mu, + \infty)^3$, for some $\mu>0$,   and $A_1, A_2$ are symmetric $3\times 3$ matrices.
\item [\bf A6.]
$b_{e0}\in   L^\infty(G)^3$,\,  $c_{0} \in  L^\infty(G)$, \  and $b_{e0,j}\geq 0$, $c_{0}\geq 0$ a.e.\ in $G$,  where  $j=1,2,3$,
\\
$ u_{e0}^1 \in H^1(G)^3$,    $u_{f0}^1\in H^2(G)^3$  and
   ${\rm div }\,  u^{1}_{f0} = 0$  in $G_f^\ve$ for $\mathcal P$-a.a. realisation $\omega \in \Omega$, \\
$u_{e0}^{\ve} \in H^1(G_e^\ve)^3$,  \,   $p^{\ve}_{e0} \in H^1(G)$,   are defined as solutions of
$$
\begin{aligned}
& {\rm div}({\bf E}^{\ve}(b_{e0,3}) \be( u^{\ve}_{e0})) = f_u  \quad  &&  \text{in } G_e^\ve, &&\\
&\Pi_\tau({\bf E}^{\ve}(b_{e0,3}) \be( u^{\ve}_{e0}) \, n ) = \ve^2 \mu \,  \Pi_\tau ( \be(u_{f0}^1) n) &&  \text{on } \Gamma^\ve,\\
& n \cdot {\bf E}^{\ve}(b_{e0,3}) \be( u^{\ve}_{e0}) \, n  =0&& \text{on } \Gamma^\ve, &&  u^\ve_{e0} = 0 \quad  \text{on } \partial G,
\\
& {\rm div} (K^{\ve}_p \nabla p^{\ve}_{e0}) = f_p \quad && \text{in } G,  &&   p^\ve_{e0} =0 \quad \text{on } \partial G,
 \end{aligned}
$$
 $\mathcal P$-a.s.,  for given  $f_u  \in L^2(G)^3$ and $f_p \in L^2(G)$,  \\
$F_p \in H^1(0,T; L^2(\partial G))$, $F_u \in H^2(0,T; L^2(\partial G))^3$.
\end{itemize}
\end{assumption}
\begin{remark} 
Under  our assumptions on $u^\ve_{e0}$ and $p^\ve_{e0}$  by the standard stochastic  homogenization arguments   we obtain
$$
\begin{aligned}
&\tilde  u^\ve_{e0} \to u_{e0}, \qquad     p^\ve_{e0} \to  p_{e0} && \text{strongly  in } L^2(G),  \\
& \be(u^\ve_{e0}) \to \be(u_{e0}) + U^0_{e,\rm{sym}} && \text{strongly  stochastically two-scale}, \; U^0_{e,\rm{sym}} \in  L^2(G; L^2_{\rm pot}(\Omega))^3,
\end{aligned}
$$
for $\mathcal P$-a.a.\ $\omega \in \Omega$, where $\tilde u_{e0}^\ve$  is an extension of  $u_{e0}^\ve$,  and  $u_{e0}\in H^1(G)^3$ and $p_{e0} \in H^1(G)$ are solutions of the corresponding macroscopic (homogenized) equations.
\\
Here the subscript {\small\rm sym} is used to emphasize that the corresponding matrix is symmetric.
\end{remark}

 Notice that in the equation for calcium  $c_f^\ve$ inside plant cells  we consider a bounded function of the  water  velocity $u_f^\ve$. This technical assumption is biologically justified, since only bounded velocities are possible inside plant cells. \\

By $\langle \cdot, \cdot \rangle_{H^1(G)^\prime, H^1}$ we shall denote the duality product between $L^2(0,T; (H^1(G))^\prime)$ and $L^2(0,T; H^1(G))$, and
$$
\langle \phi, \psi \rangle_{G_T} =  \int_0^T \int_G  \phi \, \psi \, dx dt\quad \text{ for } \quad  \phi \in L^q(0,T; L^p(G)) \text{ and } \psi \in L^{q^\prime}(0,T; L^{p^\prime} (G)),
$$
and
$$
\langle \phi, \psi \rangle_{G_T, \Omega} =  \int_0^T \int_G  \int_\Omega \phi \, \psi \, d\mathcal P(\omega) dx dt \quad \text{ for } \quad  \phi \in L^q(0,T; L^p(G\times \Omega)) \text{ and } \psi \in L^{q^\prime}(0,T; L^{p^\prime} (G\times \Omega)),
$$
where $1<p,q< + \infty$,  $1/q+ 1/q^\prime=1$ and   $1/p + 1/p^\prime =1$.

\begin{definition}
Weak solution of \eqref{eq_codif}--\eqref{exbou_co} are functions
$$
\begin{array}{l}
u_e^\ve \in [ L^2(0,T; H^1(G_e^\ve))\cap H^2(0,T; L^2(G_e^\ve))]^3,\\[1.7mm]
p_e^\ve \in L^2(0,T; H^1(G_e^\ve))\cap H^1(0,T; L^2(G_e^\ve)),\\[1.7mm]
 u_f^\ve \in [L^2(0,T; H^1(G_f^\ve))\cap H^1(0,T; L^2(G_f^\ve))]^3,\\[1.7mm]
 \hbox{\rm div}\, u_f^\ve =0 \  \hbox{\rm in }G_{f,T}^\ve, \quad \Pi_\tau u_e^\ve =\Pi_\tau u_f^\ve \;  \hbox{\rm on } \Gamma_{T}^\ve, \\[1.7mm]
 p_f^\ve \in L^2((0,T)\times G_f^\ve)
 \end{array}
 $$
for $\mathcal P$-a.a.  $\omega \in \Omega$, that satisfy the integral relation
\begin{equation}\label{weak_u_ef}
\begin{aligned}
&\langle  \rho_e \partial^2_t u^\ve_e, \phi \rangle_{G_{e,T}^\ve} + \langle{\bf E}^\ve(b_e^\ve) \e( u^\ve_e), \e(\phi) \rangle_{G_{e,T}^\ve}  +
\langle \nabla p_e^\ve, \phi \rangle_{G_{e,T}^\ve}
\\
+ &
\langle \rho_p \partial_t p^\ve_e, \psi \rangle_{G_{e,T}^\ve} + \langle K_p^\ve \nabla p^\ve_e -  \partial_t u_e^\ve, \nabla \psi \rangle_{G_{e,T}^\ve}  +\langle \partial_t u_f^\ve\cdot n , \psi \rangle_{\Gamma^\ve_T}  - \langle p_e^\ve, \eta\cdot n \rangle_{\Gamma^\ve_T}  \\
+ &  \langle  \rho_f \partial^2_t u^\ve_f, \eta \rangle_{G_{f,T}^\ve} + \mu \,  \ve^2  \langle \e(\partial_t u^\ve_f), \e(\eta) \rangle_{G_{f,T}^\ve}
 = \langle F_u, \phi \rangle_{(\partial G)_T} +  \langle F_p, \psi \rangle_{(\partial G)_T}
\end{aligned}
\end{equation}
for all $\phi \in L^2(0,T; H^1(G_e^\ve))^3$, $\psi  \in L^2(0,T; H^1(G_e^\ve))$, $\eta \in   L^2(0,T; H^1(G_f^\ve))^3$
such that $\Pi_\tau \phi =\Pi_\tau \eta $ on $\Gamma^\ve$ and ${\rm div} \, \eta =0$ in $G_{f,T}^\ve$, and functions
$$
b_e^\ve\in [L^2 (0,T; H^1(G_e^\ve))\cap L^\infty (0,T; L^2(G_e^\ve))]^3,\;
c^\ve=c_e^\ve\,  \chi\big._{G_e^\ve} + c_f^\ve\,  \chi\big._{G_f^\ve} \in L^2 (0,T; H^1(G\setminus \widetilde\Gamma^\ve))\cap L^\infty (0,T; L^2(G))
$$
that satisfy the integral relations
\begin{equation}\label{cd_one}
\begin{aligned}
\langle \partial_t b^\ve_e, \varphi_1 \rangle_{H^1(G_{e}^\ve)^\prime, H^1}+\langle D_b(b^\ve_{e,3}) \nabla b^\ve_e, \nabla\varphi_1 \rangle_{G_{e,T}^\ve}
-\langle g_b(c^\ve_e,b_e^\ve, \e(u_e^\ve)), \varphi_1 \rangle_{G_{e,T}^\ve} = \ve \langle  R(b_e^\ve), \varphi_1 \rangle_{\Gamma^\ve_T}\\ + \langle F_b(b_e^\ve), \varphi_1 \rangle_{(\partial G)_T}
\end{aligned}
\end{equation}
and
\begin{equation}\label{cd_two}
\begin{aligned}
&\; \; \langle \partial_t c^\ve_e, \varphi_2 \rangle_{H^1(G_{e}^\ve)^\prime, H^1}+\langle D_e(b^\ve_{e,3}) \nabla c^\ve_e, \nabla\varphi_2 \rangle_{G_{e,T}^\ve}
-\langle g_e(c^\ve_e,b_e^\ve, \e(u_e^\ve)), \varphi_2 \rangle_{G_{e,T}^\ve} -  \langle F_c(c_e^\ve), \varphi_2 \rangle_{(\partial G)_T}\\
& +\langle \partial_t c^\ve_f, \varphi_2 \rangle_{H^1(G_{f}^\ve)^\prime, H^1}+\langle D_f \nabla c^\ve_f, \nabla\varphi_2 \rangle_{G_{f,T}^\ve}
-\langle\mathcal G(\partial_t u^\ve_f) c_f^\ve, \nabla\varphi_2 \rangle_{G_{f,T}^\ve}-\langle g_f(c^\ve_f), \varphi_2 \rangle_{G_{f,T}^\ve}= 0
\end{aligned}
\end{equation}
for all $\varphi_1\in L^2(0,T; H^1(G^\ve_{e}))^3$ and $\varphi_2\in L^2(0,T; H^1(G\setminus\widetilde\Gamma^\ve))$, and for $\mathcal P$-a.a.\ $\omega \in \Omega$.  Moreover the initial conditions are satisfied in $L^2$-sense, i.e.\
$u_e^\ve (t) \to u^\ve_{e0}$, $\partial_t u_e^\ve(t) \to u_{e0}^1$, $p_e^\ve(0) \to p^\ve_{e0}$, $b_e^\ve(t) \to b_{e0}$, $c_e^\ve(t) \to c_0$ in $L^2(G_e^\ve)$ as $t \to 0$,
$\partial_t u_f^\ve(t)  \to u_{f0}^1$, $c_f^\ve(t) \to c_0$ in $L^2(G_f^\ve)$ as $t \to 0$, $\mathcal P$-almost sure.
\end{definition}

\noindent{\bf Examples of random geomerties}
\begin{itemize}
\item Let $\mathcal{Q}$ be a smooth domain, $\mathcal{Q}\subset(0,1)^3$, and assume that
$\gamma=\mathrm{dist}(\mathcal{Q},\partial (0,1)^3)>0$.  Let $\xi_j$  be i.i.d. random vectors in $\mathbb R^3$ such that $|\xi_j|\leq \gamma/4$, and $\eta_j$, $j\in\mathbb Z^3$, be random variables with values in the interval
$[1/2,1]$. Letting $\mathcal{Q}_j=j+\xi_j+\eta_j\mathcal{Q}$ we define
$$
G_f(\omega)=\bigcup\limits_{j\in\mathbb Z^3}\mathcal{Q}_j(\omega).
$$

\item Let $\mathcal{P}$ be a stationary ergodic point process in $\mathbb R^3$ such that\\
(i) almost surely for any two points $x_j$ and $x_k$ from $\mathcal{P}(\omega)$ the inequality $|x_j-x_k|\geq c>0$ holds  with a deterministic constant $c$;\\
(ii)
 There exists $r>0$ such that the intersection of the process with any ball of radius $r$ is a.s.
 non-empty. We then set $\mathcal{Q}_j=\{x\in\mathbb R^3\,:\, \mathrm{dist}(x,x_j)<\frac12\mathrm{dist}(x,\mathcal{P}(\omega)\setminus x_j)\}$ and define
$$
G_f(\omega)=\bigcup\limits_{j\in\mathbb Z^3}\mathcal{Q}_j(\omega).
$$
\item The last example admits the following modifications: for the same stationary point process $\mathcal{P}$
we consider the Voronoi tessellation
$$
\mathcal{Q}_j(\omega)=\{x\in\mathbb R^3\,:\, \mathrm{dist}(x,x_j)<\mathrm{dist}(x,\mathcal{P}(\omega)\setminus x_j)\}.
$$
Then $\bigcup_{j}\overline{\mathcal{Q}}_j=\mathbb R^3$ and, under the assumptions on  $\mathcal{P}$,
the diameters of the  polyhedrons $\mathcal{Q}_j$ are uniformly bounded and their boundaries are uniformly Lipschitz continuous.  \\
Given $\delta>0$ we then set
$$
G_e(\omega)=\{x\in\mathbb R^3\,:\, \mathrm{dist}(x,\bigcup\limits_j\partial\mathcal{Q}_j)<\delta\}.
$$
Notice that in this case the volume fraction of $G_e$ is of order $\delta$, if $\delta$ is sufficiently small.
This allows to model cell structures with relatively small volume fraction of cell walls and middle lamella.
\end{itemize}

\section{Main results}\label{main_results}
The main result of the paper is the derivation of the macroscopic equations for   the microscopic problem  \eqref{eq_codif}--\eqref{exbou_co} using methods of stochastic homogenization.

First we shall introduce the following notations. Denote by  $\partial_\omega^j$  the  generator of a strongly continuous group of unitary operators in $L^2(\Omega)$ associated with $\mathcal T_x$ along  $e_j$-direction, i.e.\
$$
\partial_\omega^j u(\omega) = \lim\limits_{\delta \to 0} \frac { u(\T_{\delta e_j} \omega) - u(\omega) }{\delta}.
$$
The domains of $\partial_\omega^j$, with $j=1,2,3$,  are dense in $L^2(\Omega)$. We denote $\nabla_\omega u = (\partial_\omega^1 u,  \partial_\omega^2 u, \partial_\omega^3 u)^T$ and $H^1_\T(\Omega) = \{ v \, : \,   v, \nabla_\omega v \in L^2(\Omega) \}$.
By $C_{\T} (\Omega)$ we denote the space of functions with continuous realisations and $C^1_\T(\Omega)$ defines the set of  functions from  $C_{\T} (\Omega)$ such that   $(\partial_\omega^j u) \in C_{\T} (\Omega)$, for $j=1,2,3$.

First we introduce  the spaces of potential and solenoidal vector fields:
 $$L^2_{\text{pot}} (\Omega)= \overline{ \{\nabla_\omega u \,: \, u \in C^1_{\T}(\Omega) \} }\,\mbox{ and }\,  L^2_{\text{sol}} (\Omega)= \big(L^2_{\text{pot}} (\Omega)\big)^{\perp},$$
where the closure in the definition of $L^2_{\text{pot}} (\Omega)$ is with respect to the  $L^2(\Omega)$-norm, see \cite{ZKO1994}. To introduce  correctors we also need the space of functions whose realisations are discontinuous along the surface $\widetilde \Gamma(\omega)$.  We  define
$$
L^2_{\rm pot, \Gamma}(\Omega) =\overline{\{  \nabla_x u(\mathcal T_x\omega)\big|_{x=0},\, \;\;   u(\mathcal T_x\omega) \in H_{\rm loc}^1(\mathbb R^3\setminus \widetilde \Gamma(\omega) )\cap C^1(\mathbb R^3\setminus \widetilde \Gamma(\omega) )\} }
$$
with the norm
$$
\|u\|^2=\int_\Omega\int_{[0,1]^3\setminus\widetilde \Gamma(\omega)}|\nabla_x u(\mathcal T_x\omega)|^2
dxd\mathcal P,
$$
and
$$
 L^2_{\rm sol,  \Gamma}(\Omega) = \big(L^2_{{\rm pot}, \Gamma}(\Omega)\big)^{\perp}.
 $$
We also denote
$$
C_{\T,  \Gamma} (\Omega) = \{ u \; :\;   u(\mathcal T_x\omega) \in C(\mathbb R^3\setminus \widetilde \Gamma(\omega) ) \} .
$$

We start with the definition of   effective coefficients  for macrosocpic poro-elastic equations,  which  are  obtained by deriving the macroscopic equations for the microscopic problem \eqref{equa_cla}--\eqref{exbou_co}.
The macroscopic elasticity tensor ${\bf E}^{\rm hom}= (E_{ijkl}^{\rm hom}) $ and permeability tensor  $K_p^{\rm hom}=(K_{p, ij}^{ \rm hom})$, along  with
$K_u=(K_{u, ij})$, are defined by
\begin{equation}\label{effective}
\begin{aligned}
& E_{ijkl}^{\rm hom} (b_{e,3} ) =\int_{\Omega} \big[\widetilde E_{ijkl}(\omega, b_{e,3}) + \big(\widetilde E(\omega, b_{e,3}) W^{kl}_{e, {\rm sym}} \big)_{ij}\big] \, \chi\big._{\Omega_e} d \mathcal P(\omega),\\
&K_{p, ij}^{\rm hom} = \int_{\Omega} \big[ \widetilde K_{p,ij}(\omega) + \big(\widetilde K_{p}(\omega) W_p^j\big)_i \big] \, \chi\big._{\Omega_e}\,  d \mathcal P(\omega), \\
& K_{u, ij}  =  \int_{\Omega}  \big[\delta_{ij} - \big(\widetilde K_{p}(\omega)W_u^j \big)_i\big] \, \chi\big._{\Omega_e} \, d\mathcal  P(\omega),
\end{aligned}
\end{equation}
where  $\chi\big._{\Omega_e}$ stands for the characteristic function of the set $\Omega_e$,  $W^{kl}_{e, {\rm sym}}$ denotes the symmetric part of the matrix  $W^{kl}_e$, and   $W^{kl}_e\in L^\infty(G_T; L^2_{\rm pot}(\Omega)^3)$  together with  $ W^k_p,  W^k_u\in  L^2_{\rm pot} (\Omega)$   are solutions of  cell problems
\begin{equation}\label{unit_1}
\begin{aligned}
&\int_{\Omega} \widetilde {\bf E}\big(\omega, b_{e,3}) (W^{kl}_{e, {\rm sym}}   + {\bf b}_{kl}\big)\,  \Phi \, \chi\big._{\Omega_e} d\mathcal P(\omega) = 0  && \quad \text{ for all } \Phi \in L^2_{\rm pot}(\Omega)^3, \, \text{ a.a.\  } (t,x) \in G_T,  \\
& \int_{\Omega} \widetilde K_p(\omega)\big(W_p^k + e_k\big) \,  \zeta \, \chi\big._{\Omega_e} d\mathcal P(\omega) = 0
&& \quad \text{ for all } \zeta \in L^2_{\rm pot}(\Omega), \\
& \int_{\Omega} \big(\widetilde K_p(\omega) W^k_u  -  e_k\big) \,  \zeta\, \chi\big._{\Omega_e}   d\mathcal P(\omega) = 0  && \quad \text{ for all } \zeta \in L^2_{\rm pot}(\Omega),
\end{aligned}
\end{equation}
for $k,l=1,2,3$, with ${\bf b}_{kl} = \frac 12 (e_k \otimes e_l + e_l \otimes e_k)$ and $\{ e_j\}_{j=1}^3$ is the canonical  basis of $\mathbb R^3$.

We also define $Q(\partial_t u_f)$  as
\begin{equation}\label{q_uf}
\begin{aligned}
&Q(\partial_t u_f) = \int_{\Omega}  \partial_t u_f\,  \chi\big._{\Omega_f} d \mathcal P(\omega) - \int_{\Omega} \widetilde K_p(\omega) Q_f(\omega, \partial_t u_f) \, \chi\big._{\Omega_e} \, d\mathcal P(\omega),
\end{aligned}
\end{equation}
where $Q_f(\cdot, \partial_t u_f) \in L^2_{\rm pot} (\Omega)$ is a solution of the problem
\begin{equation}\label{two-scale_qf}
\begin{aligned}
&\int_{\Omega} \big(\widetilde K_p(\omega)  \, Q_f  \, \chi\big._{\Omega_e}  + \partial_t u_f \, \chi\big._{\Omega_f}  \big)\, \zeta \, d\mathcal P(\omega)  = 0 \quad \text{ for } \zeta \in L^2_{\rm pot}(\Omega).
\end{aligned}
\end{equation}

Then the macroscopic equations for the microscopic problem \eqref{equa_cla}--\eqref{exbou_co} are formulated as follows.
\begin{theorem}\label{main_1}
A sequence of solutions $\{u_e^\ve, p_e^\ve, \partial_t u_f^\ve, p^\ve_f \}$ of the microscopic problem  \eqref{equa_cla}--\eqref{exbou_co} converges to a solution  $u_e \in H^2(0,T; L^2(G)) \cap L^2(0,T; H^1(G))$,
$p_e\in L^2(0,T; H^1(G))\cap H^1(0,T; L^2(G))$, $\partial_t u_f \in  L^2(G_T; H^1(\Omega))\cap H^1(0,T; L^2(G\times \Omega))$,
$p_f\in L^2(G_T\times \Omega)$ of the macroscopic equations
\begin{equation}\label{macro_ue}
\begin{aligned}
&\vartheta_e  \rho_e \partial_t^2 u_e - {\rm div} ( {\bf E}^{\rm hom}(b_{e,3}) \e(u_e)) + \nabla p_e+ \int_{\Omega} \partial_t^2 u_f \, \chi\big._{\Omega_f} d\mathcal P(\omega) = 0 && \text{ in } G_T, \\
&\vartheta_e  \rho_p \partial_t p_e - {\rm div} ( K_p^{\rm hom } \nabla p_e - K_ u \partial_t u_e - Q(\partial_t u_f)) =0 && \text{ in }  G_T,
\end{aligned}
\end{equation}
with boundary and initial conditions
\begin{equation}\label{macro_pe}
\begin{aligned}
& {\bf E}^{\rm hom}(b_{e,3}) \e(u_e) \,  {n} = F_u  && \text{ on } (0,T)\times  \partial G, \\
& ( K_p^{\rm hom} \nabla p_e - K_u \partial_t u_e - Q(\partial_t u_f) )\cdot {n} = F_p
&& \text{ on }  (0,T) \times \partial G, \\
& u_e(0) = u_{e0}, \quad \partial_t u_e (0) = u_{e0}^1, \quad p_e(0) = p_{e0}  && \text{ in }  G,
\end{aligned}
\end{equation}
and the  equations  for  the flow velocity
\begin{equation}\label{macro_two-scale_uf}
\begin{aligned}
&\int_{\Omega} \Big[
\rho_f \partial^2_t u_f\, \varphi + \mu \, \e_\omega(\partial_t u_f)\e_\omega (\varphi )
+ \nabla p_e\,  \varphi\Big]  \chi_{\Omega_f}  d\mathcal P(\omega)   - \int_{\Omega} P_e^1\,   \chi_{\Omega_e}\,  \varphi \,  d\mathcal P(\omega) = 0, \\
& {\rm  div}_\omega \partial_t u_f =0  \hspace{4.1 cm }  \text{ in } \; G_T\times \Omega, \\
 &\partial_t u_ f (0) = u^1_{f0} \hspace{4 cm }  \text{ in } \; G\times \Omega,
\end{aligned}
\end{equation}
$$\Pi_\tau \partial_t u_f(t,x, \T_{\widetilde x} \omega)  = \Pi_\tau \partial_t u_e(t,x) \; \; \text{ for } \; (t,x) \in G_T \; \text{ and } \; \widetilde x \in \Gamma(\omega),  \; \;   \mathcal P{\text -a.s. }   \text{ in }  \Omega,
$$
 and
\begin{equation}\label{expres_Pe1}
P_e^1(t,x,\omega)
= \sum_{k=1}^3 \partial_{x_k} p_e(t,x) W_p^k(\omega) + \partial_t u_e^k(t,x) W_u^k(\omega)
+  Q_f(\omega, \partial_t u_f),
\end{equation}
for all $\varphi \in L^2(G_T; H^1(\Omega))^3$,  with  ${\rm div}_\omega \varphi = 0 $ in $G_T\times \Omega$,  and $\Pi_\tau \varphi(t,x, \T_{\widetilde x} \omega) =0$ for $(t,x) \in G_T$, $\widetilde x \in \Gamma(\omega)$ and  $\mathcal P$-a.s.~in  $\Omega$.
\\
Here    $\e_\omega(\psi) =\big(1/2(\partial_\omega^j \psi_l + \partial_\omega^l \psi_j)\big)_{j,l=1,2,3}$  denotes  a symmetric gradient for $\psi \in H^1(\Omega)^3$,  $\vartheta_e = \int_{\Omega} \chi_{\Omega_e} d\mathcal P(\omega)$, and ${\rm div}_\omega \psi = \partial_\omega^1 \psi_1 + \partial_\omega^2 \psi_2 + \partial_\omega^3 \psi_3$.
 \end{theorem}

\emph{Remark.} Notice  that the equations for correctors \eqref{unit_1} and  \eqref{two-scale_qf}, as well as  problem   \eqref{macro_two-scale_uf} for $\partial_t u_f$ are formulated in the weak form as integral identities.  This is due to the fact that the equations  are define on $\Omega_e\subset \Omega$ and $\Omega_f \subset \Omega$, respectively,  and have strong formulation only for  $\mathcal P$-a.a.\  realisations $\omega \in \Omega$. \\

The  homogenized coefficients  in  reaction-diffusion-convection equations that will be obtained by deriving macroscopic equations for  microscopic problem \eqref{eq_codif}, \eqref{exbou_co}, are defined as
\begin{equation}
\begin{aligned}
& D_{b, {\rm eff}}^{ij}(b_{e,3}) = \int_{\Omega} \Big[ D^{ij}_b(b_{e,3}) + (D_b(b_{e,3})\,  w^j_b  )_i\Big]\chi_{\Omega_e}  d\mathcal P(\omega),
\\
& D^{ij}_{\rm eff} (b_{e,3})= \int_{\Omega} \Big[ D^{ij}(b_{e,3}, \omega) +  \big(D(b_{e,3}, \omega)\, w^j \big)_i \Big]   d \mathcal P  (\omega),
\end{aligned}
\end{equation}
where $D(b_{e,3}, \omega) = D_e(b_{e,3}) \chi\big._{\Omega_e}(\omega) + D_f \chi\big._{\Omega_f}(\omega)$  for $\omega \in \Omega$,  with $w^j_b\in L^2_{\rm pot}(\Omega)$ and $w^j\in L^2_{\rm pot,  \Gamma}(\Omega)$ are  solutions of the cell problems
\begin{equation}\label{unit_cell_prob_b}
\begin{aligned}
 &\int_{\Omega}  D_b(b_{e,3}) ( w^j_b + e_j) \, \zeta \, \chi\big._{\Omega_e}   d\mathcal P(\omega) = 0 \quad \text{ for all }\;  \zeta \in L^2_{\rm pot}(\Omega),
\end{aligned}
\end{equation}
and
\begin{equation}\label{unit_cell_prob_c}
\begin{aligned}
& \int_{\Omega}  D(\omega, b_{e,3}) (w^j + e_j) \, \eta \,  d\mathcal P(\omega) = 0 \quad  \text{ for all } \; \eta \in L^2_{\rm pot,  \Gamma}(\Omega) .
\end{aligned}
\end{equation}
 The effective velocity is defined  by
 $$u_{\rm eff}(t,x) = \int_{\Omega}\Big( \mathcal G(\partial_t u_f) - D_f  Z(t,x,\omega) \Big) \chi\big._{\Omega_f} \, d\mathcal P(\omega), $$
 where $Z \in L^\infty(G_T; L^2_{\rm pot} (\Omega))$ satisfies
\begin{equation}\label{unit_cell_prob_cf}
\begin{aligned}
& \int_{\Omega} ( D_f  Z - \mathcal G(\partial_t u_f))\, \zeta \, \chi\big._{\Omega_f} \, d\mathcal P(\omega)  = 0  \quad \text{ for all } \zeta \in L^2_{\rm pot} (\Omega), \; \text{for a.a.\ } (t,x) \in G_T.
\end{aligned}
\end{equation}

\begin{theorem}\label{th:reactions}
A sequence of solutions of microscopic problem   \eqref{eq_codif}, \eqref{exbou_co} converges to a solution $b_e, c \in L^2(0,T; H^1(\Omega))$, with $\partial_t b_e , \partial_t c \in L^2(0,T; (H^{1}(\Omega))^\prime)$,  of the macroscopic equations
\begin{equation}\label{macro_concentration}
\begin{aligned}
\vartheta_e  \partial_t b_e - {\rm div} ( D_{b,{\rm eff}}(b_{e,3}) \nabla b_e ) &= \hspace{-0.1 cm }  \int_{\Omega} g_b(c, b_e,   \mathbb{U}(b_e, \omega)  \e(u_e))  \chi\big._{\Omega_e}  d\mathcal P(\omega) + \hspace{-0.1 cm } \int_\Omega R(b_e) d\boldsymbol{\mu}(\omega)    && \text{in } G_T, \\
\hspace{-0.4 cm }  \partial_t c - {\rm div} ( D_{\rm eff} (b_{e,3})\nabla c  - u_{\rm eff}  c)& =\vartheta_f g_f(c)+  \int_{\Omega} g_e(c, b_e,   \mathbb U(b_e, \omega)\e(u_e))  \chi\big._{\Omega_e}  d\mathcal P(\omega)  && \text{in } G_T,\\
 D_{b,\rm eff} (b_{e,3})\nabla b_e   \cdot n& = F_b(b_e)  \quad && \text{on } (\partial G)_T, \\
 ( D_{\rm eff}(b_{e,3}) \nabla c  - u_{\rm eff}  c)  \cdot n &= F_c(c)  \quad && \text{on }  (\partial G)_T, \\
 b_e(0,x) = b_{e0}(x), \quad & c(0,x)= c_0(x)  \quad && \text{in }  G,
\end{aligned}
\end{equation}
where $\vartheta_j = \int_\Omega \chi\big._{\Omega_j}(\omega)  \, d \mathcal P(\omega)$, for $j=e,f$, and
$$\mathbb U(b_e, \omega) =\left\{\mathbb U_{klij}(b_e, \omega)\right\}_{k,l,i,j=1}^3 = \left\{ b_{kl}^{ij} + W^{ij}_{e,{\rm sym}, kl}\right\}_{k,l,i,j=1}^3, $$
with  $W^{ij}_e$ being solutions of cell problems \eqref{unit_1} and ${\bf b}_{kl} = (b_{kl}^{ij})_{i,j=1}^3$,  where ${\bf b}_{kl} = e_k\otimes e_l$. 
   \end{theorem}
  Here $\boldsymbol{\mu}$ is the Palm measure of the random measure $\mu_\omega$ of surfaces $\Gamma(\omega)$, see e.g.\ \cite{Daley_VereJones_1988} for the definition of Palm measure.

\section{A priori estimates}\label{apriori_1}

Considering assumptions on $G^\ve_j$, with $j=e,f$, in the same way as in the periodic case \cite{Andrey_Mariya}, for $\mathcal P$-a.a.\ realisations $\omega \in \Omega$,  we  obtain  the  existence, uniqueness  and   a priori estimates, uniform in $\ve$,  for solutions of  microscopic problem \eqref{eq_codif}--\eqref{exbou_co}.

\begin{lemma}  Under Assumption~\ref{assumptions1}
there exists a unique weak  solution of  microscopic problem  \eqref{eq_codif}--\eqref{exbou_co}.
\end{lemma}
\begin{proof}[Proof Sketch.]
For each realisation $\omega$ the proof of the existence and uniqueness results   follows the same steps as the proof of Theorem 7 in \cite{Andrey_Mariya}.
\end{proof}

\begin{lemma}
Under Assumptions~\ref{assumptions1}  solutions of microscopic problem  \eqref{eq_codif}--\eqref{exbou_co} satisfy \textit{a priori } estimates for elastic displacement $u^\ve_e$, pressure  $p_e^\ve$, and fluid flow velocity  $\partial_t u^\ve_f$
\begin{equation}\label{estim_u_p_u}
\begin{aligned}
&\| u_e^\ve\|_{L^\infty(0,T; H^1(G_e^\ve))} + \|\partial_t u_e^\ve\|_{L^2(0,T; H^1(G_e^\ve))} + \| \partial^2_t u_e^\ve \|_{L^2(G_{e,T}^\ve)} \leq C , \\
&\| p_e^\ve\|_{L^\infty(0,T; H^1(G_e^\ve))}  + \| \partial_t p_e^\ve \|_{L^2(G_{e,T}^\ve)} \leq C, \\
& \| \partial_t u_f^\ve\|_{L^\infty(0,T; L^2(G_f^\ve))} + \| \partial^2_t u_f^\ve\|_{L^2(G_{f,T}^\ve)}
 + \ve \|\nabla \partial_t u_f^\ve\|_{L^2(G_{f,T}^\ve)} + \| p_f^\ve \|_{L^2(G_{f,T}^\ve)} \leq C,
\end{aligned}
\end{equation}
and for the concentration of calcium $c^\ve_e$ and $c_f^\ve$ and densities  of pectins and calcium-pectin cross-links $b_e^\ve$ we obtain
\begin{equation}\label{estim_b_c}
\begin{aligned}
&\| b_e^\ve\|_{L^2(0,T; H^1(G_e^\ve))} + \|b_e^\ve\|_{L^\infty(0,T; L^\infty(G_{e}^\ve))} + \ve^{1/2} \| b_e^\ve\|_{L^2(\Gamma_T^\ve)}\leq C , \\
&\| c_j^\ve\|_{L^2(0, T; H^1(G_j^\ve))}  + \| c_j^\ve\|_{L^\infty(0, T; L^4(G_j^\ve))} \leq C, \; \; & j = e,f,
\end{aligned}
\end{equation}
and
\begin{equation}\label{estim_time_h}
\begin{aligned}
& \| \theta_h b_e^\ve -  b^\ve_e\|_{L^2((0, \tilde T) \times G_e^\ve)} +  \| \theta_h c_j^\ve -  c^\ve_j\|_{L^2((0, \tilde T) \times G_j^\ve)}  \leq C h^{1/4}, \; \quad j=e,f,
\end{aligned}
\end{equation}
for $\tilde T \in (0, T-h]$ and for $\mathcal P$-a.a.\ $\omega \in \Omega$,  where   the constant $C$ is independent of $\ve$
and $\theta_h v(t,x) = v(t+h, x)$ for $(t,x) \in (0, T-h]\times G_j^\ve$, with $j=e,f$.
\end{lemma}
\begin{proof}
For  $\mathcal P$-a.a.\  realisations $\omega\in \Omega$  the proof of the a priori estimates follows the same lines as in  \cite[Lemma 6]{Andrey_Mariya}.
\end{proof}

We shall denote  $c^\ve (t,x)=  c_e^\ve(t,x) \, \chi\big._{G_{e}^\ve} +
c_f^\ve(t,x) \chi\big._{G_{f}^\ve}.$

Using the assumptions on the random microscopic structure of $G^\ve_e$ and $G_f^\ve$ we obtain the following extension results for functions defined on $G^\ve_e$ and on a subdomain  $\widetilde G_{ef}^\ve \subset G$, which will be specified below.
\begin{lemma}\label{extension}
\begin{itemize}
\item[(i)] There exist extensions $\overline b_e^\ve$ and $\overline c_e^\ve $ of  $b_e^\ve$ and $c_e^\ve$, respectively,  from $L^2(0,T; H^1(G_e^\ve))$  to $L^2(0, T; H^1(G))$ such that
\begin{equation}\label{estim_ext_1}
\| \overline b_e^\ve \|_{L^2(G_T)} \leq  C \|b_e^\ve \|_{L^2(G^\ve_{e,T})}, \quad \| \nabla \overline b_e^\ve \|_{L^2(G_T)} \leq  C \|\nabla b_e^\ve \|_{L^2(G^\ve_{e,T})},
\end{equation}
\begin{equation}\label{estim_ext_11}
\| \overline c_e^\ve \|_{L^2(G_T)} \leq  C \|c_e^\ve \|_{L^2(G^\ve_{e,T})}, \quad \| \nabla \overline c_e^\ve \|_{L^2(G_T)} \leq  C \|\nabla c_e^\ve \|_{L^2(G^\ve_{e,T})}.
\end{equation}
\item[(ii)] There exists an extension $\overline c^\ve $ of  $c^\ve$ from  $L^2(0,T; H^1(\widetilde G_{ef}^\ve))$  to $L^2(0,T; H^1(G))$ such that
\begin{equation}\label{estim_ext_2}
\begin{aligned}
& \| \overline c^\ve \|_{L^2(G_T)} \leq  C \left(\|c_e^\ve \|_{L^2(G^\ve_{e,T}\cap\widetilde  G_{ef, T}^\ve)} +  \|c_f^\ve \|_{L^2(G^\ve_{f,T}\cap\widetilde  G_{ef, T}^\ve)}\right) , \quad \\
& \| \nabla \overline c^\ve \|_{L^2(G_T)} \leq  C \left(\|\nabla c_e^\ve \|_{L^2(G^\ve_{e,T}\cap\widetilde  G_{ef, T}^\ve)}+ \|\nabla c^\ve_f \|_{L^2(G^\ve_{f,T}\cap\widetilde  G_{ef, T}^\ve)}\right).
\end{aligned}
\end{equation}
Here   $\widetilde G^\ve_{ef}=G \setminus \widetilde G^\ve$, with  $\widetilde G^\ve= \widetilde \Gamma^\ve_{\ve\sigma}(\omega) \cap \widetilde G^\ve_{e}$, where  $\widetilde \Gamma^\ve_{\ve\sigma}(\omega)$ is a $\ve\sigma$-neighbourhood of $\widetilde \Gamma^\ve$   for  $\mathcal P$-a.a.\  realisations $\omega \in \Omega$ and  $0<\sigma< d_{\rm dim}/4$, with  $d_{\rm min}$ being  the minimal distance between connected components of $G_f(\omega)$.
\end{itemize}
\end{lemma}

\begin{proof}
The uniform boundedness of the diameter of  cell walls and cell  interiors, independent on realisations $\omega \in \Omega$,  implies the existence of the corresponding extension operators, see \cite{ACDP}
\end{proof}
Extensions for $u^\varepsilon_e$ and $p^\varepsilon_e$ are defined in the similar way as for 
$b^\varepsilon_e$.

\begin{lemma}\label{cor_estim_extend}
For extensions of $b_e^\ve$, $c^\ve_e$,  $u_e^\ve$, $p_e^\ve$ from $G^\ve_{e,T}$  to $G_T$ and $c^\ve$  from $\widetilde G^\ve_{ef,T}$  to $G_T$  (denoted again by  $b_e^\ve$, $c^\ve_e$,   $u_e^\ve$,   $p_e^\ve$, and $c^\ve$)  we have the following estimates
\begin{equation}\label{estim_extend}
\begin{aligned}
& \| u_e^\ve\|_{H^1(0,T; H^1(G))} + \| \partial^2_t u_e^\ve \|_{L^2(G_T)}+\| p_e^\ve\|_{L^\infty(0,T; H^1(G))}  + \| \partial_t p_e^\ve \|_{L^2(G_T)} \leq C, \\
 &\| b_e^\ve\|_{L^2(0,T; H^1(G))} + \| c^\ve_e\|_{L^2(0, T; H^1(G))}  + \| c^\ve\|_{L^2(0, T; H^1(G))}   \leq C, \\
& \| \theta_h b_e^\ve -  b^\ve_e\|_{L^2((0, \tilde T) \times G)} +  \| \theta_h c^\ve_e -  c^\ve_e\|_{L^2((0, \tilde T) \times G)}+  \| \theta_h c^\ve -  c^\ve\|_{L^2((0, \tilde T) \times G)}  \leq C h^{1/4},
\end{aligned}
\end{equation}
where the constant $C$ is independent of $\ve$.
An extension of  $\partial_t u_f^\ve$  from $G_{f,T}^\ve$ to $G_T$, constructed below and denoted again by $\partial_t u_f^\ve$ satisfies the following estimates
\begin{equation}\label{estim_extend_2}
\begin{aligned}
& \| \partial_t u_f^\ve\|_{L^\infty(0,T; L^2(G))} + \| \partial^2_t u_f^\ve\|_{L^2(G_T)}
 + \ve \|\nabla \partial_t u_f^\ve\|_{L^2(G_T)}  \leq C,
 \end{aligned}
\end{equation}
where the constant $C$ is independent of $\ve$.
Also we have that
\begin{equation}\label{estim_pf}
\|\widetilde p^\ve\|_{L^2(G_T)} + \|p^\ve_f\|_{L^2((0,T)\times G^\ve_f)}  \leq C, \quad \text{ where }  \;  \widetilde p^\ve =\begin{cases}  p_f^\ve & \text{ in } \; \;  (0,T)\times  G_f^\ve, \\
p_e^\ve & \text{ in } \; \;  (0,T)\times( G \setminus  G_f^\ve),
\end{cases}
\end{equation}
and the constant $C$ does not depend on $\ve$.
\end{lemma}
\begin{proof}
  The estimates for $b_e^\ve$, $c^\ve_e$, $c^\ve$,  $u_e^\ve$, and $p_e^\ve$  follow directly from estimates \eqref{estim_u_p_u}--\eqref{estim_time_h}, Lemma~\ref{extension}, and  the  linearity of extension considered  in Lemma~\ref{extension}.

Using geometrical assumptions on $G_f(\omega)$, for $\mathcal P$-a.a.\  $\omega\in \Omega$,
we can extend $\partial_t u_f^\ve$  from $G_f^\ve$ to $G$ in the following way. For  each connected component $G_{f,j}(\omega)$  of $G_f(\omega)$, with $j \in \mathbb N$, we can consider a $\sigma$-neighbourhood $G^\sigma_{f,j}(\omega)$ of $G_{f,j}(\omega)$, where $\sigma = d_{\rm min}/4$ and $d_{\rm min}$ is the minimal distance between $G_{f,j}(\omega)$ for $j \in \mathbb N$.  Then  since $\partial_t u^\ve_f \in L^2(0,T; H^1(G^\ve_f))$, i.e.\ $\partial_t u_f^\ve \in L^2(0,T; H^{1/2}(\Gamma^\ve))$,   there exists  $\partial_t \widetilde u_f^j \in L^2(0,T; H^1(G^\sigma_{f, j}(\omega)\setminus G_{f,j}(\omega))$ satisfying the problem
\begin{equation}
\begin{aligned}
& {\rm div} _y\partial_t \widetilde u_f^j =  0 \quad && \text{ in } G^\sigma_{f, j}(\omega)\setminus G_{f,j}(\omega), \\
&  \partial_t \widetilde u_f^j =  \partial_t  u_f^\ve(t, \ve y) \quad&&  \text{ on } \Gamma_j(\omega), \\
& \partial_t \widetilde u_f^j =   0 \quad  && \text { on } \partial  G^\sigma_{f, j}(\omega)
\end{aligned}
\end{equation}
for $\mathcal P$-a.a.\ realisations $\omega \in \Omega$ and $j \in \mathbb N$, see e.g.\ \cite[Theorem 2.4, Lemma 2.4]{Temam}.
Each  $\partial_t \widetilde u_f^j $ we extend by zero to $G_e(\omega) \setminus G^\sigma_{f, j}(\omega)$.  Considering a scaling $y=x/\ve$ in $\partial_t\widetilde u_f^j$  and collecting all $\partial_t \widetilde u_f^j$ for $j \in \mathbb N$ we obtain an extension $\partial_t \overline u^\ve_f$ of $\partial_t u_f^\ve$ from $G_f^\ve$ to $G$ such that
$\partial_t \overline u^\ve_f  \in L^2(0,T; H^1(G))$ and
\begin{equation}\label{extend_uf}
\begin{aligned}
&{\rm div}   \partial_t \overline u_f^\ve  =0 \quad \text{ in } G, \\
& \| \partial_t u_f^\ve\|_{L^2(G_T)}+ \ve \|\nabla \partial_t u_f^\ve \|_{L^2(G_T)} \leq C,
 \end{aligned}
\end{equation}
where the constant $C$ is independent of    $\ve$.

 Similar to the periodic case to show the a priori estimates for $p_f^\ve$ we consider   the first and third equations  in \eqref{equa_cla}    and use the a priori estimates for $u_e^\ve$,  $p_e^\ve$, and $\partial_t u_f^\ve$  to obtain
\begin{equation}\label{estim_pf_2}
\begin{aligned}
\langle p_f^\ve, {\rm div} \, \phi \rangle_{G_{f,T}^\ve} +  \langle p_e^\ve, {\rm div }\, \phi   \rangle_{G_{e, T}^\ve} = \langle \ve^2 \mu \, \be(\partial_t u_f^\ve),  \be(\phi)  \rangle_{G_{f,T}^\ve} +  \langle \rho_f  \partial^2_t u_f^\ve, \phi\rangle_{G_{f, T}^\ve}
\\
 + \langle \rho_e  \partial^2_t u_e^\ve, \phi\rangle_{G_{e, T}^\ve}
+  \langle {\bf E}^\ve(b^\ve_{e,3}) \be(u_e^\ve), \be(\phi) \rangle_{G_{e, T}^\ve}+\langle p_e^\ve\,  n - F_u, \phi \rangle_{(\partial G)_T} \\
\leq C \|\phi\|_{L^2(0,T; H^1(G))^3},
\end{aligned}
\end{equation}
with  $\phi \in L^2(0,T; H^1(G))^3$.
Here we used  the extension of  $p_e^\ve$  from $G_e^\ve$ to $G$, see Lemma~\ref{extension}, and  the trace estimate
$\|p_e^\ve\|_{L^2((0,T)\times\partial G)}\leq C_1 \|p_e^\ve\|_{L^2(0,T; H^1(G))} \leq C_2 \|p_e^\ve\|_{L^2(0,T; H^1(G_e^\ve))}$.

For any $q \in L^2(G_T)$ there exists $\phi \in  L^2(0,T; H^1(G))^3$ satisfying
$${\rm div} \, \phi =q \quad \text{ in } \; G, \qquad  \phi \cdot n = \frac 1{|\partial G|}\int_G q(\cdot,x) dx  \quad \text{ on } \;  \partial G
$$ and $\| \phi\|_{L^2(0,T; H^1(G))^3}\leq C \|q\|_{L^2(G_T)}$.
Thus
 using \eqref{estim_pf}, the definition of the $L^2$-norm,  and the a priori estimates for  $p^\ve_e$ we obtain
$$
\|\widetilde p^\ve\|_{L^2(G_{T})} \leq C  \quad \text{ and } \quad \|p_f^\ve\|_{L^2((0,T)\times G_f^\ve)} \leq C,
$$
where  the constant $C$ is independent of $\ve$.
\end{proof}

\section{Convergence results}\label{convergence}
From a priori estimates derived in Lemma~\ref{cor_estim_extend} we obtain corresponding  strong and  stochastic two-scale convergences for a subsequence of solutions of microscopic problem \eqref{eq_codif}--\eqref{exbou_co}.  First we recall the definition of the stochastic two-scale convergence introduced in \cite{Zhikov_Piatnitski_2006}.

\begin{definition}\label{def_t_s}
 Let $G$ be a domain in $\mathbb R^3$,  $\T_x$ be an  ergodic dynamical system, and $\widetilde \omega$ be a ``typical realisation''.  Then, we say that a sequence $\{ v^\ve \} \subset L^2(0, T; L^2(G))$
converges  stochastically  two-scale to   $v\in L^2(G_T; L^2(\Omega, d \mathcal P))$ if
\begin{equation}\label{estim_t-s-def}
\limsup\limits_{\ve \to 0} \int_0^T\int_G |v^\ve(t, x)|^2 \,  d x \,  dt < \infty
\end{equation}
and
\begin{eqnarray}\label{stoch_two_scale}
&&\lim\limits_{\ve \to 0} \int_0^T\int_{G} v^\ve(t, x) \varphi(t,x) \psi(\T_{x/\ve}\widetilde \omega)\, dx dt \\ && \hspace{ 4 cm } =
\int_0^T\int_G\int_\Omega v(t, x, \omega) \varphi(t,x) \psi(\omega)  \, d \mathcal P(\omega)  dx  dt \nonumber
\end{eqnarray}
for all   $\varphi \in C^\infty_0([0, T)\times G)$ and $\psi \in L^2(\Omega)$.
\end{definition}

As a ``typical realisation'' we denote such  realisation $\omega\in\Omega$ that  Birkhoff's theorem is satisfied for $\T_x\omega$, i.e.
$$
\lim\limits_{\ell\to \infty} \frac 1 { \ell^3|A|} \int_{\ell A} g(\T_x \omega) \, dx  = \int_\Omega g(\omega) \, d\mathcal P(\omega)
$$
$\mathcal P$-a.s.\  for all bounded Borel sets $A$, $|A|>0$, and all $g(\omega) \in C(\Omega)$.
Let us note that realisations are typical $\mathcal P$-a.s., see e.g.\ \cite{Zhikov_Piatnitski_2006}.

Using  compactness  properties of  stochastic two-scale convergence, see  \cite{Zhikov_Piatnitski_2006},  we obtain the following result.
\begin{lemma}\label{convergence_u_p}
There exist functions $u_e \in H^1(0,T; H^1(G))\cap H^2(0,T; L^2(G))$, $p_e \in L^2(0,T; H^1(G))\cap H^1(0,T; L^2(G))$,
$U_e^1, \partial_t U_e^1 \in L^2(G_T; L^2_{\rm pot} (\Omega))^{3}$,  $P_e^1 \in L^2(G_T; L^2_{\rm pot} (\Omega))$,  and  $\partial_t u_f\,  \chi\big._{\Omega_f}$, $\nabla_\omega \partial_t u_f \,  \chi\big._{\Omega_f}$, $\partial_t^2 u_f  \, \chi\big._{\Omega_f}$,  $p_f \,  \chi\big._{\Omega_f}   \in L^2(G_T\times \Omega)$,  such that, up to a subsequence,
\begin{equation}\label{convergences_1}
\begin{aligned}
& u_e^\ve \rightarrow u_e && \text{ strongly in } H^1(0,T; L^2(G)), \\
& p_e^\ve \rightarrow p_e && \text{ strongly in } L^2((0,T)\times G), \\
&  \partial^2_t u_e^\ve  \rightharpoonup \partial^2_t u_e, \; \quad  \partial_t p_e^\ve  \rightharpoonup  \partial_t p_e&& \text{ stochastically two-scale} , \\
& \nabla u_e^\ve  \rightharpoonup\nabla u_e +U_e^1 && \text{ stochastically  two-scale} , \\
& \nabla p_e^\ve \rightharpoonup  \nabla p_e + P_e^1 && \text{ stochastically two-scale} ,
\end{aligned}
\end{equation}
and for fluid velocity and pressure we have
\begin{equation}\label{convergences_2}
\begin{aligned}
&\chi\big._{G^\ve_f}  \partial_t u_f^\ve  \rightharpoonup \chi\big._{\Omega_f} \partial_t u_f && \text{ stochastically two-scale} , \\
&\ve \chi\big._{G^\ve_f} \nabla \partial_t u_f^\ve  \rightharpoonup  \chi\big._{\Omega_f} \nabla_\omega \partial_t u_f && \text{ stochastically two-scale} , \\
& \chi\big._{G^\ve_f}  \, p_f^\ve  \rightharpoonup  \chi\big._{\Omega_f} \, p_f && \text{ stochastically two-scale}.
\end{aligned}
\end{equation}
\end{lemma}
\begin{proof}
The estimates \eqref{estim_extend},    the compactness of the  embedding of $H^1(0,T; L^2(G))\cap L^2(0,T; H^1(G))$ in $L^2(G_T)$,  and the compactness theorem for stochastic  two-scale convergence, see e.g.\ \cite{Zhikov_Piatnitski_2006}, yield the convergence results in \eqref{convergences_1}.

For the extension of $u_f^\ve$ from $G_f^\ve$ to $G$ we have the stochastic two-scale convergence of $\partial_t u_f^\ve \rightharpoonup \partial_t u_f$ and $\ve\nabla \partial_t u_f^\ve  \rightharpoonup \nabla_\omega \partial_t u_f$, with  $\partial_t u_f,  \nabla_\omega \partial_t u_f \in L^2(G_T\times \Omega)$, respectively.   Additionally we have that  $U_e^1 \, \chi\big._{\Omega_e} $, $ P_e^1\,  \chi\big._{\Omega_e}$,  $ \partial_t u_f \, \chi\big._{\Omega_f}$,  and $\nabla_\omega \partial_t u_f\,  \chi\big._{\Omega_f} $ do not depend on the extension of  $u_e^\ve$, $p_e^\ve$ from $G_e^\ve$ to $G$ and of $\partial_t u_f^\ve$ from $G_f^\ve$ to $G$.
 The estimate and  definition of $\widetilde p^\ve$  in \eqref{estim_pf} and  \eqref{estim_pf_2} ensure the stochastic two-scale convergence of   $\chi\big._{G^\ve_f} p_f^\ve$.
\end{proof}

In the following lemma, we shall use the same notation for $b_e^\ve$, $c^\ve_e$ and their extensions from $G_e^\ve$ to $G$, whereas the extension for $c^\ve$ from $\widetilde G_{ef}^\ve$ to $G$ will be denoted by $\overline c^\ve$.
\begin{lemma}\label{convergence_b_c}
There exist functions $b_e, c  \in L^2(0,T; H^1(G))$,  $b_e \in L^\infty(0,T; L^\infty(G))$,  $c\in L^\infty(0,T; L^4(G))$, and correctors
$B_e^1 \in L^2(G_T; L^2_{\rm pot}(\Omega))$  and  $C^1 \in L^2(G_T; L^2_{{\rm {pot}},  \Gamma} (\Omega))$,
 such that, up to a subsequence,
\begin{equation}\label{convergences_11}
\begin{aligned}
& b_e^\ve \rightarrow b_e, \;  \; \; c^\ve \to c && \text{ strongly in } L^2(G_T),  \\
&\nabla b_e^\ve \rightharpoonup \nabla b_e + B_e^1 && \text{ stochastically two-scale},  \\
& \nabla c^\ve \rightharpoonup \nabla c + C^1 && \text{ stochastically two-scale}, \; \; \; \; \text{ as } \; \ve \to 0.
\end{aligned}
\end{equation}
\end{lemma}

\begin{proof}
The estimates in  \eqref{estim_extend}, together with compactness results for stochastic two-scale convergence, see \cite{Zhikov_Piatnitski_2006}, ensure that  for every ``typical'' realisation $\widetilde \omega \in \Omega$ there exist  $b_e, c_e,  c  \in L^2(0,T; H^1(G))$ and  $B_e^1$, $C_e^1$,  $\overline C^1\in L^2(G_T; L^2_{{\rm pot}}(\Omega))$,  such that
$\nabla b_e^\ve  \rightharpoonup \nabla b_e + B_e^1$,  $\nabla c_e^\ve \rightharpoonup \nabla c_e + C_e^1$, and $\nabla \overline c^\ve \rightharpoonup \nabla  c + \overline C^1$  stochastically  two-scale.
Estimates \eqref{estim_b_c} and \eqref{estim_extend}   and    the compactness of the embedding of $ H^1(G)$ in $L^2(G)$,  together with the Kolmogorov compactness theorem, see e.g.\  \cite{Brezis, Necas},  yield the strong convergence $b_e^\ve \to b_e$,  $c^\ve_e \to c_e$ and $\overline c^\ve \to c$  in $L^2(G_T)$ for $\mathcal P$-a.a.\  realisations $\widetilde \omega \in \Omega$.  Since  $G_{e,T}^\ve\cap \widetilde G_{ef, T}^\ve \neq \emptyset$,    $c_e^\ve(t,x) = \overline c^\ve(t,x)$ for a.a.\ $(t,x) \in G_{e,T}^\ve \cap \widetilde G_{ef, T}^\ve$, and $c_e$ and $c$ are independent of $\omega\in \Omega$,  we obtain that $c_e(t,x) = c(t,x)$  for a.a.\  $(t,x) \in G_T$ and $\mathcal P$-a.s in $\Omega$.

From the  estimates for $c^\ve= c_e^\ve \chi\big._{G^\ve_e}+ c_f^\ve \chi\big._{G^\ve_f}$ in \eqref{estim_b_c}  we obtain that   there exists $C^1 \in L^2(G_T; L_{{\rm pot}, \Gamma}(\Omega))$ such that $\nabla c^\ve   \rightharpoonup \nabla c + C^1$ stochastically   two-scale.
\end{proof}

\section{Derivation of macroscopic equations for flow velocity  and elastic deformations.}\label{macro_elasticity}
 To show the convergence of boundary terms  we shall prove the relation between convergence with respect to $\mathcal P$ in $G$ and Palm measure $\boldsymbol{\mu}$ on the oscillating surfaces $\Gamma^\ve$.

\begin{definition} \cite{Daley_VereJones_1988}
The Palm measure of the random stationary measure $\mu_\omega$ is the measure $\boldsymbol{\mu}$
on $(\Omega, \mathcal F)$ defined as
$$
\boldsymbol{\mu}(F) = \int_\Omega \int_{\mathbb R^3} \chi\big._{[0,1)^3}(x) \chi\big._F(\mathcal T_x \omega)\,  d \mu_\omega(x) \, d \mathcal P(\omega) \quad \text{ for }  \; F \in \mathcal F.
$$
\end{definition}

\begin{lemma}\label{lemma_trace_ineq}
For $u \in H^1(\Omega, \mathcal P)$ we have that $u \in L^2(\Omega, \boldsymbol{\mu})$, where $\boldsymbol{\mu}$ is the Palm measure of the random stationary measure  $\mu_\omega$ of surfaces $\Gamma(\omega)$  for realisations $\omega \in \Omega$, and the  embedding is continuous.
\end{lemma}
\begin{proof}
Consider $u \in H^1(\Omega, \mathcal P)$ and a random stationary measure $\mu_\omega$ given by  {  $d\mu_\omega(x)= {\bf 1}_{\Gamma(\omega)} d \sigma(x)$,  where $d\sigma(x)$ is the standard surface measure}.
By $\boldsymbol{\mu}$ we denote the Palm measure of the random stationary  measure $\mu_\omega$.
Let $\mathcal{Q}_\rho$ be the ball in $\mathbb R^3$ of radius $\rho$ centered at the origin.
Since $u\in H^1(\Omega, \mathcal P)$, then a.s. $u(\mathcal T_{x} \omega) \in H_{\rm loc}^1(\mathbb R^3)$.
Under our
assumptions by the trace theorem there exist $\delta>0$ and $C>0$ such that
\begin{equation}\label{trace}
\begin{aligned}
\int_{\Gamma(\omega)\cap \mathcal{Q}_\rho} |u(\mathcal T_x\omega)|^2 d \sigma(x)  \leq C
 \int_{{\mathcal Q}_{\rho+\delta}} |u(\mathcal T_x\omega)|^2 d x +
 C \int_{{\mathcal Q}_{\rho+\delta}} |\nabla u(\mathcal T_x\omega)|^2 d x
\end{aligned}
\end{equation}
$\mathcal P$-a.s.\ in $\Omega$.
We divide the left- and the right-hand sides of this relation by $\rho^3$ and pass to the limit,
as $\rho\to\infty$. By the Birkhoff theorem we obtain
$$
\int_{\Omega}|u(\omega)|^2\,d\boldsymbol{\mu}\leq C\Big[\int_{\Omega}|u(\omega)|^2\,d\mathcal{P}+
\int_{\Omega}|\nabla_\omega u(\omega)|^2\,d\mathcal{P}\Big].
$$
This yields the desired statement.
\end{proof}

\begin{proof}[Proof of Theorem~\ref{main_1}]  To derive macroscopic equations for the system of poro-elastic and Stokes equations, first we consider as test functions  in \eqref{weak_u_ef} the following functions

\begin{itemize}
\item $\phi(t,x)=\ve  \phi_1(t,x)\phi_2(\T_{x /\ve}  \widetilde \omega)$, \, $\phi_1 \in C^1_0(G_T)$,  \, $\phi_2 \in C^1_\T(\Omega)^3$
\item $\psi(t,x)= \ve \psi_1(t,x) \psi_2(\T_{x/\ve}  \widetilde \omega )$,  \,  $\psi_1 \in C^1_0(G_T)$, \,  $\psi_2 \in C^1_\T(\Omega)$,   $\eta_1 \in C^1_0(G_T)$
\item $\eta(t,x)= \ve \eta_1(t,x) \eta_2(\T_{x/\ve}  \widetilde \omega )$,   \,   $\eta_2 \in C^1_\T(\Omega)^3$, and $\phi_1(t,x)\Pi_\tau  \phi_2(\T_{\widetilde x }\widetilde \omega ) = \eta_1(t,x) \Pi_\tau \eta_2(\T_{\widetilde x }  \widetilde \omega)$
\end{itemize}
for $(t,x) \in G_T$,  $\widetilde x  \in \Gamma(\widetilde \omega)$, and $\mathcal P$-a.a.\ realisations $\widetilde \omega \in \Omega$.
To apply stochastic two-scale convergence of $u_e^\ve$, $p_e^\ve$, and $\partial_t u_f^\ve$,  we rewrite the boundary integrals  over $\Gamma^\ve$ in the weak formulation \eqref{weak_u_ef}  as volume integrals
\begin{equation}\label{weak_u_ef_2}
\begin{aligned}
&\langle  \rho_e \partial^2_t u^\ve_e, \phi \chi\big._{G_e^\ve} \rangle_{G_T} + \langle{\bf E}^\ve(b_e^\ve) \e( u^\ve_e), \e(\phi)  \chi\big._{G_e^\ve}  \rangle_{G_T}  +
\langle \nabla p_e^\ve, \phi\,   \chi\big._{G_e^\ve}  \rangle_{G_T}
+
\langle \rho_p \partial_t p^\ve_e, \psi \,  \chi\big._{G_e^\ve}\rangle_{G_T} \\
& + \langle K_p^\ve \nabla p^\ve_e -  \partial_t u_e^\ve, \nabla \psi  \chi\big._{G_e^\ve}  \rangle_{G_T}  - \langle \partial_t u_f^\ve , \nabla \psi  \chi\big._{G_f^\ve}  \rangle_{G_T}  + \langle \nabla p_e^\ve, \eta \,   \chi\big._{G_f^\ve}   \rangle_{G_T}  +  \langle p_e^\ve, {\rm div}\,\eta  \,  \chi\big._{G_f^\ve} \rangle_{G_T}   \\
& +  \langle  \rho_f \partial^2_t u^\ve_f, \eta \,   \chi\big._{G_f^\ve} \rangle_{G_T} + \mu \,  \ve^2  \langle \e(\partial_t u^\ve_f), \e(\eta)\,   \chi\big._{G_f^\ve}  \rangle_{G_T}  -
\langle p_f^\ve, \text{\rm div}\,\eta \,   \chi\big._{G_f^\ve} \rangle_{G_T}  \\
& = \langle F_u, \phi \rangle_{(\partial G)_T} +  \langle F_p, \psi \rangle_{(\partial G)_T}.
\end{aligned}
\end{equation}
Here we have used the relation ${\rm div} \partial_t u_f^\ve =0$ in $G_{f,T}^\ve$ and the fact that $\chi\big._{G_j^\ve}(x, \omega)= \chi\big._{\Omega_j}(\mathcal T_{x/\ve} \omega)$  $\mathcal P$-a.s.\  in $\Omega$, where $j=e,f$.
Using the convergence results in Lemma~\ref{convergence_u_p}   and passing to the limit $\ve\to 0$ we obtain
\begin{equation}\label{eq:limit_1}
\begin{aligned}
\langle \widetilde{\bf E}(\omega, b_{e,3}) (\e( u_e) + U_{e, {\rm sym}}^1),  \phi_1 \e_\omega(\phi_2)\,  \chi\big._{\Omega_e} \rangle_{G_{T}\times \Omega}
 +  \langle \widetilde K_p(\omega) ( \nabla p_e + P_e^1)- \partial_t u_e,  \psi_1 \nabla_\omega \psi_2\,  \chi\big._{\Omega_e} \rangle_{G_{T}\times \Omega} \\
-  \langle  \partial_t u_f, \psi_1\nabla_\omega \psi _2 \chi\big._{\Omega_f}  \rangle_{G_T\times \Omega}
+  \langle p_e,  \eta_1  {\rm div}_ \omega \eta_2 \, \chi\big._{\Omega_f}  \rangle_{G_T\times \Omega}
 -
 \langle  p_f, \eta_1 \text{div}_\omega \eta_2\, \chi\big._{\Omega_f}  \rangle_{G_{T}\times \Omega}
 = 0.
\end{aligned}
\end{equation}
Letting first $\psi_1\equiv 0$ and $\eta_1\equiv 0$ and then $\phi_1\equiv 0$ and $\eta_1\equiv 0$ we obtain   the equations for the correctors  $U_e^1$ and $P_e^1$, i.e.\
\begin{equation}\label{correctos_ue}
\begin{aligned}
\big \langle \widetilde {\bf E}(\omega, b_{e,3}) ({\rm \e}( u_e) + U^1_{e, {\rm sym}}) \chi\big._{\Omega_e} , \phi_1 \, \e_\omega (\phi_2)  \big \rangle_{G_T\times \Omega} = 0,
\end{aligned}
\end{equation}
and
\begin{equation}\label{correctos_pe}
\begin{aligned}
\langle \big(\widetilde K_p(\omega) ( \nabla p_e +P_e^1)- \partial_t u_e\big) \chi\big._{\Omega_e} - \partial_t u_f\,  \chi\big._{\Omega_f} , \psi_1\,  \nabla_\omega \psi_2   \rangle_{G_T\times \Omega}  &= 0 .
\end{aligned}
\end{equation}
From \eqref{eq:limit_1}  considering $\phi_1\equiv 0$ and $\psi_1\equiv 0$ also yields
\begin{equation*}
p_f \, \chi\big._{\Omega_f} = p_e \,  \chi\big._{\Omega_f}   \qquad \text{ in } G_T\times \Omega.
\end{equation*}
Next, choosing in \eqref{weak_u_ef} test functions of the form $(\phi(t, x), \psi(t, x), \eta(t, x, x/\ve))$,  where
\begin{itemize}
\item  $\phi\in C^\infty(\overline G_T)^3$ and  $\psi\in C^\infty(\overline G_T)$,
\item   $\eta(t,x, x/\ve) =  \eta_1(t,x) \,  \eta_2(\T_{x/\ve}\omega)$,  where  $\eta_1\in C^\infty(\overline G_T)$,  $\eta_2\in  C^1_\T(\Omega)^3$,   with   ${\rm div}_\omega \eta_2 =0$ for $\mathcal P$-a.a. $\omega \in \Omega$,  and
  $ \eta_1(t,x) \, \Pi_\tau \eta_2(\T_{\widetilde x }  \omega) = \Pi_\tau \phi(t,x)$  for $(t,x) \in G_T$,  $\widetilde x \in \Gamma(\omega)$, and $\mathcal P$-a.s. in $\Omega$,
\end{itemize}
we  obtain
\begin{equation}\label{weak_form_lim}
\begin{aligned}
&\langle \rho_e  \partial^2_t u^\ve_e,  \phi\,  \chi\big._{G_e^\ve}  \rangle_{G_T} + \langle {\bf E}^\ve(b^\ve_e) \e( u^\ve_e),  \e(\phi) \,  \chi\big._{G_e^\ve}  \rangle_{G_T}  +
\langle \nabla p_e^\ve,  \phi \,  \chi\big._{G_e^\ve}  \rangle_{G_T}  \\ &+
\langle \rho_p  \partial_t p^\ve_e,  \psi \,  \chi\big._{G_e^\ve} \rangle_{G_T}
+ \langle K_p^\ve \nabla p^\ve_e - \partial_t u_e^\ve,  \nabla \psi  \,   \chi\big._{G_e^\ve} \rangle_{G_T}
-\langle \partial_t u_f^\ve ,   \nabla \psi \,  \chi\big._{G_f^\ve} \rangle_{G_T}\\
&+ \langle \nabla p_e^\ve,  \eta_1\, \eta_2 \rangle_{G_{T}}-  \langle \nabla p_e^\ve,  \eta_1\, \eta_2 \,  \chi\big._{G_e^\ve} \rangle_{G_T}
+  \langle p_e^\ve, {\rm div}_x  \eta_1\, \eta_2 \,  \chi\big._{G_f^\ve} \rangle_{G_T} \\
& + \langle  \rho_f \partial^2_t u^\ve_f,  \eta_1\, \eta_2 \,  \chi\big._{G_f^\ve}  \rangle_{G_T}
+ \mu\,   \ve^2  \langle \e(\partial_t u^\ve_f), [ \e(\eta_1)\eta_2 +  \ve^{-1} \eta_1 \e_\omega(\eta_2)]\,  \chi\big._{G_f^\ve}  \rangle_{G_T}  \\
& -
\langle p_f^\ve,  \text{div}_x \eta_1\, \eta_2 \,   \chi\big._{G_f^\ve}  \rangle_{G_T}    = \langle F_u, \phi \rangle_{(\partial G)_T} + \langle F_p, \psi \rangle_{(\partial G)_T}.
\end{aligned}
\end{equation}
Letting $\ve \to 0$ and using the stochastic two-scale and strong convergences of $u_e^\ve$ and $p_e^\ve$, the strong convergence of $b_e^\ve$,  and  the stochastic  two-scale convergence of $\partial_t u_f^\ve$  we obtain
\begin{equation}\label{macro_111}
\begin{aligned}
&\langle  \rho_e \partial^2_t u_e ,  \phi \, \chi\big._{\Omega_e}\rangle_{G_{T}, \Omega}
 + \langle \widetilde{\bf E}(\omega, b_{e,3}) \big({\rm \e}( u_e)+ U_{e, {\rm sym}}^1 \big), {\rm \e}(\phi) \,  \chi\big._{\Omega_e} \rangle_{G_{T}, \Omega}  +
\langle \nabla p_e+ P_e^1, \phi\,  \chi\big._{\Omega_e} \rangle_{G_{T},  \Omega}
\\
& +
\langle\rho_p  \partial_t p_e,  \psi \,  \chi\big._{\Omega_e}  \rangle_{G_{T},  \Omega}
 + \langle  \widetilde K_p(\omega)( \nabla p_e+  P_e^1)  - \partial_t u_e,   \nabla \psi \, \chi\big._{\Omega_e}\rangle_{G_{T}, \Omega}  -  \langle \partial_t u_f, \nabla \psi \,  \chi\big._{\Omega_f} \rangle_{G_T, \Omega}  \\
 &  +
\langle \nabla p_e,  \eta_1 \eta_2\, \chi\big._{\Omega_f}  \rangle_{G_{T}, \Omega}
+  \langle  P_e^1,  \eta_1 \eta_2  \rangle_{G_{T}\times \Omega} -  \langle  P_e^1,   \eta_1 \eta_2 \,  \chi\big._{\Omega_e} \rangle_{G_{T}, \Omega}\\
& + \langle  \rho_f \partial^2_t u_f,  \eta_1\eta_2 \,  \chi\big._{\Omega_f} \rangle_{G_{T}, \Omega}
 + \mu    \langle \e_\omega(\partial_t u_f),   \eta_1\e_\omega(\eta_2)\, \chi\big._{\Omega_f}  \rangle_{G_{T},  \Omega}     = \langle F_u, \phi \rangle_{(\partial G)_T} + \langle F_p, \psi \rangle_{(\partial G)_T}.
\end{aligned}
\end{equation}
Here we used the fact that $ \chi\big._{\Omega_f} \, p_f= \chi\big._{\Omega_f} \, p_e$ in $G_T\times \Omega$.  Since $P_e^1 \in L^2(G_T; L^2_{\rm pot}(\Omega))$ and  $\eta_1\in C(\overline G_T)$, $\eta_2 \in L^2_{\rm sol}(\Omega)$ we obtain that
$$
  \langle  P_e^1,  \eta_1\, \eta_2  \rangle_{G_{T},  \Omega} =0.
$$

The stochastic  two-scale convergence of $\partial_t u_f^\ve$  and the fact that  $\partial_t u_f^\ve$  is divergence-free  in $G_T$ (we identify here $\partial_t u_f^\ve$ with its  extension  constructed in Lemma~\ref{cor_estim_extend}) imply
\begin{equation*}
\begin{aligned}
0= \lim\limits_{\ve\to 0}\langle \text{div}\, \partial_t u_f^\ve, \ve \eta(t, x, x/\ve) \rangle_{G_T} =
- \lim\limits_{\ve\to 0}  \langle  \partial_t u_f^\ve, \ve \nabla_x\eta + \nabla_\omega \eta  \rangle_{G_T}\\  =
- \langle  \partial_t u_f, \nabla_\omega \eta  \rangle_{G_T\times \Omega}  = \langle  \text{div}_\omega \partial_t  u_f,   \eta \, \rangle_{G_T\times \Omega}.
\end{aligned}
\end{equation*}
Thus $\text{div}_{\omega}\, \partial_t u_f =0$ a.e.\ in $G_T$ and $\mathcal P$-a.s.\  in  $\Omega$.

Choosing $\phi\equiv 0$ and $\psi\equiv 0$, and taking  $\eta= \eta_1 \eta_2$, where
$\eta_1 \in C^1_0(G_T)$ and  $\eta_2 \in C^1_\T(\Omega)^3$,  with $\text{ div}_\omega \eta_2 = 0$ and  $\Pi_\tau \eta_2(\mathcal T_x \omega) =0$ on $\Gamma(\omega)$  $\mathcal P$-a.s.\ in  $\Omega$,  we conclude that $\partial_t u_f$ is a solution to  problem \eqref{macro_two-scale_uf}.
Taking   $\eta = \eta_1 \eta_2$, with  $\eta_2 = {\rm const}$  and $\eta_1 \in C^1_0(0,T; C^1(\overline G))^3$ as a test function in    \eqref{macro_two-scale_uf}   yields
\begin{equation}\label{expr_Pe1}
\langle P_e^1,  \eta_1 \chi\big._{\Omega_e} \rangle_{G_T, \Omega} = \langle \rho_f \partial^2_t  u_f + \nabla p_e, \eta_1 \chi\big._{\Omega_f} \rangle_{G_T, \Omega}.
\end{equation}

 Next we have to determine  the boundary conditions for tangential components of $\partial_t u_f$ on $\Gamma(\omega)$ for $\mathcal P$-a.a. $\omega \in \Omega$.
From a priori estimates for $\partial_t u_e^\ve$ and $\partial_t u_f^\ve$ we have that
$$
\begin{aligned}
&\ve \|\partial_t u^\ve_e\|^2_{L^2(\Gamma_{T}^\ve)} \leq C_1\big(  \|\partial_t u^\ve_e\|^2_{L^2(G_{e,T}^\ve)} +\ve^2 \|\nabla \partial_t u^\ve_e\|^2_{L^2(G_{e,T}^\ve)}\big) \leq C_2, \\
&\ve \|\partial_t u^\ve_f\|^2_{L^2(\Gamma_{T}^\ve)} \leq C_3\big(  \|\partial_t u^\ve_f\|^2_{L^2(G_{f,T}^\ve)} +\ve^2 \|\nabla \partial_t u^\ve_f\|^2_{L^2(G_{f,T}^\ve)}\big) \leq C_4,
\end{aligned}
$$
where the constants $C_j$, with $j=1,2,3,4$, are independent of $\ve$.
Thus  using Lemmata~\ref{lemma_trace_ineq}   and \ref{converg_bound_Gamma} and the fact that $\partial_t u_f \in L^2(G_T; H^1(\Omega))$ and $\partial_t u_e \in L^2(0,T; H^1(G))$ we obtain
$$
\begin{aligned}
\int_{G_T} \int_\Omega \Pi_\tau  \partial_t u_f(t,x,\omega) \psi_1(t,x)\psi_2( \omega) \, d\boldsymbol{\mu} dx dt = \lim\limits_{\ve \to 0} \ve \int_{\Gamma^\ve_T} \Pi_\tau  \partial_t u^\ve_f (t,x) \psi_1(t,x) \psi_2(\T_{x/\ve} \widetilde\omega) \,   d\sigma^\ve dt
\\ = \lim\limits_{\ve \to 0} \ve \int_{\Gamma^\ve_T} \Pi_\tau  \partial_t u^\ve_e(t,x) \psi_1(t,x) \psi_2(\T_{x/\ve} \widetilde\omega)  \,  d\sigma^\ve dt= \int_{G_T} \int_\Omega \Pi_\tau  \partial_t u_e(t,x) \psi _1(t,x) \psi_2(\omega)\,  d\boldsymbol{\mu} dx dt
\end{aligned}
$$
for $\psi_1 \in C^1_0(G_T)$, $\psi_2 \in C^1(\Omega)$ and typical realisations $\widetilde \omega \in \Omega$.  Thus for each typical realisation $\widetilde \omega \in \Omega$ we have
$$
\Pi_\tau \partial_t u_f = \Pi_\tau \partial_t u_e \qquad \text{on}  \quad G_T\times \Gamma(\widetilde \omega).
$$


Considering first $\phi\in C^\infty_0(G_T)^3$,  $\psi\in C^\infty_0(G_T)$,  and then
$\phi\in C^\infty(\overline G_T)^3$,  $\psi\in C^\infty(\overline G_T)$,  and using  equality  \eqref{expr_Pe1} together with
\begin{equation}\label{corrector_U}
U_e^1= \sum_{k,l=1}^3 \e(u_e(t,x))_{kl}  W_e^{kl}(t,x,\omega),
\end{equation}
where $W_e^{kl}$ are solutions of the first equations in \eqref{unit_1},  yield the macroscopic equations for $u_e$:
\begin{equation}
\begin{aligned}
\vartheta_ e \rho_e \partial_t^2 u_e  -{\rm div} \big( {\bf E}^{\rm hom} (b_{e,3}) \e(u_e)\big) +  \nabla p_e + \int_{\Omega} \rho_f \partial^2_t u_f \, \chi_{\Omega_f}  d\mathcal P(\omega) & = 0  \qquad&&  \text{ in } G_T, \\
 {\bf E}^{\rm hom} (b_{e,3}) \e(u_e)\, n &  = F_u  && \text{ on } (\partial G)_T,
\end{aligned}
\end{equation}
where  ${\bf E}^{\rm hom}$ is defined by \eqref{effective},  as well as the equation
\begin{equation}\label{macro_pe_11}
\begin{aligned}
\vartheta_e \rho_p \partial_t p_e  -{\rm div}\Big( \int_{\Omega} \Big[\big(\widetilde K_p(\omega)(\nabla p_e +  P_e^1)   - \partial_t u_e \big) \chi_{\Omega_e} -  \partial_t u_f \, \chi_{\Omega_f} \Big] \, d\mathcal P(\omega) \Big)&= 0 && \text{ in }  G_T, \\
 \qquad \Big( \int_{\Omega}\Big[ \big(\widetilde K_p(\omega)(\nabla p_e +P_e^1)   -  \partial_t u_e\big) \chi_{\Omega_e} - \partial_t u_f\, \chi_{\Omega_f} \Big] \, d\mathcal P(\omega)  \Big)\cdot n  &= F_p &&
 \text{ on } (\partial G)_T,
\end{aligned}
\end{equation}
together with  problem \eqref{correctos_pe} for $ P_e^1$.
The structure of the problem \eqref{correctos_pe} suggests that $P_e^1$ should be of the form
\begin{equation}\label{struct_Pe11}
\begin{aligned}
P_e^1(t,x, \omega)= \sum_{k=1}^3 \frac {\partial p_e}{\partial x_k} (t,x) \, W^k_p(\omega) + \sum_{k=1}^3 \partial_{t} u_e^k(t,x) \,  W^k_u(\omega) + Q_f(\omega, \partial_t  u_f),
\end{aligned}
\end{equation}
where
$W^k_p$ and $W^k_u$ are solutions of cell problems \eqref{unit_1}, and $Q_f$ is a solution of problem \eqref{two-scale_qf}.
Substituting the right-hand side of \eqref{struct_Pe11} for $P_e^1$ in \eqref{macro_pe_11} we obtain the macroscopic equations for $p_e$ in   \eqref{macro_ue},
where   $K_p^{\rm hom}$ and $K_u$ are  defined in~\eqref{effective}.
\end{proof}

\section{Strong stochastic  two-scale convergence of $\e(u^\ve_e)$, $\nabla p_e^\ve$,  and $\partial_t u_f^\ve$. }\label{strong_convergence}
Due to the presence of nonlinear functions depending on  $\e(u^\ve_e)$ and
$\partial_t u_f^\ve$ in equations for $b_e^\ve$, $c_e^\ve$, and $c_f^\ve$,  in order to derive the macroscopic equations for $b_e$ and $c$ we have to show that $\e(u^\ve_e)$ and $\partial_t u_f^\ve$ converge stochastically two-scale strongly.
\begin{lemma}\label{lem_strong_two-scale} For a   subsequences of $\{u^\ve_e\}$, $\{ p_e^\ve \}$  and $\{\partial_t u_f^\ve\}$  as in Lemma~\ref{convergence_u_p} (denoted again by $\{u^\ve_e\}$, $\{ p_e^\ve \}$, and $\{\partial_t u_f^\ve\}$) we have
\begin{equation}
\begin{aligned}
&\chi_{G^\ve_e} \e(u_e^\ve) \to \chi_{\Omega_e} \,  (\e(u_e) + U_{e, {\rm sym}}^1) && \text{ strongly stochastic two-scale} , \\
&\chi_{G^\ve_e}  \nabla p_e^\ve \to   \chi_{\Omega_e} \, (\nabla p_e + P_e^1) && \text{ strongly stochastic two-scale} , \\
&\chi_{G^\ve_f} \partial_t u_f^\ve \to  \chi_{\Omega_f} \partial_t u_f && \text{ strongly stochastic two-scale}.
\end{aligned}
\end{equation}
\end{lemma}

\begin{proof} Similar to the periodic case \cite{Andrey_Mariya},  to show the  strong stochastic two-scale convergence of  $\e(u^\ve_e)$,
$p_e^\ve$, and
$\partial_t u_f^\ve$ we prove  the convergence of  the energy related to the equations for $u_e^\ve$, $p_e^\ve$, and $\partial_t u_f^\ve$.
We consider a monotone decreasing function   $\varrho:\mathbb R_+\to \mathbb R_+$, e.g.\ $\varrho(t)= e^{-\gamma t}$  for $t \in \mathbb R_{+}$,  and define    the energy functional for the microscopic problem \eqref{equa_cla}--\eqref{exbou_co} as
\begin{equation}
\begin{aligned}
\hspace{-0.4 cm } \mathcal E^\ve(u_e^\ve, p_e^\ve, \partial_t u^\ve_f) &= \frac 12\rho_e  \|\partial_t u_e^\ve(s)\varrho(s) \|^2_{L^2(G_e^\ve)} -  \rho_e \langle \varrho^\prime(\cdot)\varrho(\cdot)\, \partial_t u_e^\ve,  \partial_t u_e^\ve  \rangle_{G_{e,s}^\ve} \\
& + \frac 12  \langle {\bf E}^\ve(b^\ve_{e,3}) \e(u^\ve_e)(s), \e(u_e^\ve)(s)\varrho^2(s) \rangle_{G_{e}^\ve} \\
 & -  \frac 12 \left\langle \big( 2  \varrho^\prime(\cdot) \varrho(\cdot) {\bf E}^\ve(b^\ve_{e,3})  +  \varrho^2(\cdot) \partial_t {\bf E}^\ve(b^\ve_{e,3})\big)  \e(u^\ve_e),  \e(u^\ve_e)\right\rangle_{G_{e, s}^\ve}
\\ & +\frac 12 \rho_p \|p_e^\ve(s)\varrho(s)\|^2_{L^2(G_e^\ve)} - \rho_p \langle \varrho^\prime(\cdot)\varrho(\cdot),  |p_e^\ve|^2 \rangle_{G_{e,s}^\ve}
+ \langle K^\ve_p \nabla p_e^\ve \varrho(\cdot), \nabla p_e^\ve \varrho(\cdot)\rangle_{G_{e,s}^\ve}
\\
& + \frac 12 \rho_f  \|\partial_t u_f^\ve(s)\varrho(s)\|^2_{L^2(G_f^\ve)}  - \rho_f  \langle \varrho^\prime(\cdot)\varrho(\cdot)\, \partial_t u_f^\ve,  \partial_t u_f^\ve  \rangle_{G_{f,s}^\ve}
+ \mu \| \ve \varrho(\cdot) \e(\partial_t u_f^\ve)\|^2_{L^2(G_{f,s}^\ve)}
\end{aligned}
\end{equation}
for $s \in (0,T)$ and $\mathcal P$-a.a.\ $\omega \in \Omega$. Considering  $\partial_t u_e^\ve\, \varrho^2$, $p_e^\ve\, \varrho^2$, and $\partial_t u_f^\ve\, \varrho^2$ as test functions in  \eqref{weak_u_ef}  yields  the   equality
\begin{equation}
\begin{aligned}
\mathcal E^\ve(u_e^\ve, p_e^\ve, \partial_t u^\ve_f) = \frac 12\rho_e \|\partial_t u_e^\ve(0)\|^2_{L^2(G_e^\ve)} +
 \frac 12  \langle {\bf E}^\ve(b^\ve_{e,3}) \e(u^\ve_e)(0), \e(u_e^\ve)(0) \rangle_{G_{e}^\ve} +
\frac 12 \rho_f \|\partial_t u_f^\ve(0)\|^2_{L^2(G_f^\ve)} \\
+ \frac 12\rho_p \|p_e^\ve(0)\|^2_{L^2(G_e^\ve)} +  \langle F_u,  \partial_t u_e^\ve\, \varrho^2 \rangle_{(\partial G)_T} +  \langle F_p, p_e^\ve\, \varrho^2 \rangle_{(\partial G)_T}.
\end{aligned}
\end{equation}
Due to assumptions on  $\widetilde{\bf E}$ and $\partial_t \widetilde{\bf E}$ there exists such   $\gamma>0$  that
 $$
  \big(2 \gamma  {\widetilde{\bf E}}_1(\omega, \mathcal K(\eta)) -  \partial_t {\widetilde{\bf E}}_1(\omega, \mathcal K(\eta)) \big) A \cdot A  \geq 0  \; \text{ for all symmetric matrices } A \text{ and } \eta \in \mathbb R, \text{ and  $\mathcal P$-a.a.\ } \omega \in \Omega.
 $$

The weak stochastic two-scale convergence of $({\bf E}^\ve(b^\ve_{e,3}))^{1/2} \e( u_e^\ve)$,   $(2\gamma {\bf E}^\ve(b^\ve_{e,3}) -  \partial_t {\bf E}^\ve(b_{e,3}^\ve))^{1/2} \e( u_e^\ve)$,   and $(K^\ve_p)^{1/2} \nabla p^\ve_e$, as $\ve \to 0$,  and the lower-semicontinuity of the norm ensure
\begin{equation}
\begin{aligned}
\hspace{-0.2 cm }  &\phantom{+\, } \frac {\rho_e} 2 \|\partial_t u_e(s)\varrho(s)\,  \chi\big._{\Omega_e} \|^2_{L^2(G\times \Omega)} + \gamma \rho_e  \|\partial_t u_e\, \varrho \chi\big._{\Omega_e}\|^2_{L^2(G_s\times  \Omega)} \\
\hspace{-0.2 cm }  & + \frac 12  \langle \widetilde{\bf E}(\omega, b_{e,3}) \varrho^2(s) \big(\e(u_e(s))+ U_{e, {\rm sym}}^1(s)\big)\,  \chi\big._{\Omega_e}, \e(u_e(s)) + U_{e, {\rm sym}}^1(s)\rangle_{G,  \Omega}
 \\
\hspace{-0.2 cm }  & +  \frac 12 \langle \varrho^2 \big(2 \gamma\,   \widetilde{\bf E}(\omega, b_{e,3}) -   \partial_t\widetilde{\bf E}(\omega, b_{e,3})\big) (\e(u_e)+ U_{e, {\rm sym}}^1) \chi\big._{\Omega_e},  \e(u_e)+  U_{e, {\rm sym}}^1\rangle_{G_{s}, \Omega}
\\
\hspace{-0.2 cm } & + \frac {\rho_p}2  \|p_e(s)\varrho(s) \chi_{\Omega_e}\|^2_{L^2(G\times \Omega)}
+ \gamma \rho_p  \|p_e  \varrho  \chi_{\Omega_e}\|^2_{L^2(G_{s}\times \Omega)}
+ \langle  \varrho^2 \widetilde K_p(\omega) (\nabla p_e+ P_e^1) \chi\big._{\Omega_e}, \nabla p_e + P_e^1\rangle_{G_{s}, \Omega}
\\
\hspace{-0.2 cm } & +\frac { \rho_f}2   \|\partial_t u_f(s)\varrho(s) \chi\big._{\Omega_f}\|^2_{L^2(G\times\Omega)}
+  \gamma \rho_f  \|\partial_t u_f \, \varrho \chi\big._{\Omega_f}\|^2_{L^2(G_{s}\times \Omega)}
+ \mu  \| \e_\omega(\partial_t u_f)\, \varrho\,  \chi_{\Omega_f}\|^2_{L^2(G_s\times \Omega)}
\\
\hspace{-0.2 cm }  & \leq \liminf \limits_{\ve \to 0} \mathcal E^\ve(u_e^\ve, p_e^\ve, \partial_t u^\ve_f) \leq \limsup\limits_{\ve \to 0} \mathcal E^\ve(u_e^\ve, p_e^\ve, \partial_t u^\ve_f)
 = \frac { \rho_e} 2 \|\partial_t u_e(0)\, \chi\big._{\Omega_e} \|^2_{L^2(G\times  \Omega)} \\
\hspace{-0.2 cm }   &+ \frac 12  \big\langle \widetilde{\bf E}(\omega, b_{e,3})\big( \e(u_e)(0)+ U_{e, \rm{sym}}^0\big) \, \chi\big._{\Omega_e}, \e(u_e)(0)+ U_{e, \rm{sym}}^0  \big\rangle_{G,  \Omega}
 + \frac { \rho_p} 2 \|p_e(0) \chi\big._{\Omega_e}\|^2_{L^2(G\times \Omega)}  \\
 \hspace{-0.2 cm }  &+
\frac {\rho_f}2 \|\partial_t u_f(0) \chi\big._{\Omega_f}\|^2_{L^2(G\times\Omega)}  + \langle F_u,  \partial_t u_e\,  \varrho^2 \rangle_{(\partial G)_s, \Omega}+ \langle F_p,  p_e \, \varrho^2\rangle_{(\partial G)_s, \Omega}.
\end{aligned}
\end{equation}
Here we also used the strong convergence of   $b_e^\ve$  and the stochastic  two-scale convergence of $\nabla p_e^\ve$, $\e(u^\ve_e)$, $\partial_t u_f^\ve$,  and $\ve \e(\partial_t u_f^\ve)$.
Considering the limit equations for  $u_e$, $U_e^1$, $p_e$, $P_e^1$, and $\partial_t u_f$ and taking $(\partial_t u_e \, \varrho^2, p_e\,  \varrho^2, \partial_t u_f \, \varrho^2)$ as a test function  yield
\begin{equation}
\begin{aligned}
& \frac {\rho_e}2 \| \partial_t u_e(s) \varrho(s)\chi\big._{\Omega_e} \|^2_{L^2(G\times \Omega)}  - \frac {\rho_e}2 \| \partial_t u_e(0) \chi\big._{\Omega_e}\|^2_{L^2(G\times\Omega)}+ \gamma \rho_e  \| \partial_t u_e \, \varrho \,  \chi\big._{\Omega_e} \|^2_{L^2(G_{s}\times \Omega)}
\\
&+ \langle \widetilde{\bf E}(\omega, b_{e,3}) \big(\e( u_e)+ U_{e, {\rm sym}}^1)\big),  \e(\partial_t u_e) \, \varrho^2  \chi\big._{\Omega_e} \rangle_{G_{s}, \Omega}  +
\langle  \nabla p_e+ P_e^1,  \partial_t u_e \chi\big._{\Omega_e}  \rangle_{G_{s}, \Omega}
\\
& +
\frac {\rho_p}2 \|p_e(s) \varrho(s)\,  \chi\big._{\Omega_e}\|^2_{L^2(G\times \Omega)}
- \frac {\rho_p}2 \| p_e(0) \, \chi\big._{\Omega_e} \|^2_{L^2(G\times \Omega)}+ \gamma  \rho_p \| p_e\,  \varrho\,  \chi\big._{\Omega_e}\|^2_{L^2(G_{s}\times \Omega)} \\
& +\big \langle \big[\widetilde K_p(\omega)( \nabla p_e+ P_e^1)  -  \partial_t u_e \big] \chi\big._{\Omega_e} - \partial_t u_f \, \chi\big._{\Omega_f} ,  \nabla p_e\,  \varrho^2 \big\rangle_{G_{s}, \Omega} \\
& + \frac {\rho_f}2 \|\partial_t u_f(s)\, \varrho(s)\chi\big._{\Omega_f}\|^2_{L^2(G\times \Omega)} -  \frac {\rho_f}2 \|\partial_t u_f(0) \chi\big._{\Omega_f} \|^2_{L^2(G\times \Omega)} +  \gamma \rho_f \| \partial_t u_f\,  \varrho\,  \chi\big._{\Omega_f}\|^2_{L^2(G_{s}\times \Omega)} \\
& + \mu    \langle \e_\omega(\partial_t u_f),   \e_\omega(\partial_t u_f)\varrho^2 \chi\big._{\Omega_f} \rangle_{G_{s},  \Omega}  + \langle \nabla p_e, \partial_t u_f\,  \varrho^2 \, \chi\big._{\Omega_f} \rangle_{G_s, \Omega} -  \langle   P_e^1 \chi\big._{\Omega_e},  \partial_t u_f   \varrho^2\rangle_{G_{s},  \Omega} \\
& = \langle F_u, \partial_t u_e\, \varrho^2 \rangle_{(\partial G)_s} +  \langle F_p, p_e\, \varrho^2 \rangle_{(\partial G)_s}
\end{aligned}
\end{equation}
for $s \in (0,T)$.
Taking $P_e^1$ as a test function in the equation for $P_e^1$ yields
\begin{equation}
\begin{aligned}
   \langle   P_e^1,  \partial_t u_f \,   \varrho^2\, \chi\big._{\Omega_f}\rangle_{G_{s}, \Omega} =  \langle \widetilde K_p(\omega)( \nabla p_e+ P_e^1)  - \partial_t u_e,  P_e^1 \varrho^2\, \chi\big._{\Omega_e} \rangle_{G_{s}, \Omega}.
\end{aligned}
\end{equation}
Since $P_e^1\in L^2(G_T; L^2_{\rm pot} (\Omega))$ and $\partial_t u_f  \in L^2(G_T; L^2_{\rm sol} (\Omega))$ we obtain
$$
\langle   P_e^1,  \partial_t u_f  \,  \varrho^2\rangle_{G_{s}, \Omega}  = 0  \; \; \text{ and } \; \;
  \langle   P_e^1,  \partial_t u_f \,   \varrho^2\, \chi\big._{\Omega_e}\rangle_{G_{s}, \Omega} = -   \langle   P_e^1,  \partial_t u_f \,   \varrho^2\, \chi\big._{\Omega_f}\rangle_{G_{s}, \Omega} .
$$
Considering   equation \eqref{correctos_ue} for the corrector $U_e^1$ and taking   $\partial_t U_e^1 \, \varrho^2$ as a test function imply
\begin{equation}
\begin{aligned}
& \langle \widetilde{\bf E}(\omega, b_{e,3}) \big(\e( u_e)+ U_{e, {\rm sym}}^1\big),  \e(\partial_t u_e) \, \varrho^2\, \chi\big._{\Omega_e} \rangle_{G_{s}, \Omega} \\
&= \langle \widetilde{\bf E}(\omega, b_{e,3}) \big(\e( u_e)+ U_{e, {\rm sym}}^1\big),  (\e(\partial_t u_e) +  \partial_t U_{e, {\rm sym}}^1)\, \varrho^2\, \chi\big._{\Omega_e} \rangle_{G_{s},  \Omega}  \\
& =
\frac 12 \left \langle \widetilde{\bf E}(\omega, b_{e,3}) \big(\e( u_e(s))+ U_{e, {\rm sym}}^1(s)\big)\varrho^2(s)\, \chi\big._{\Omega_e} ,  \e(u_e(s)) +  U_{e, {\rm sym}}^1(s) \right\rangle_{G,  \Omega} \\
&  -
\frac 12 \left \langle \widetilde{\bf E}(\omega, b_{e,3}) \big(\e( u_e(0))+ U_{e, {\rm sym}}^1\big)\chi\big._{\Omega_e},   \e(u_e(0)) +  U_{e, {\rm sym}}^1\right \rangle_{G, \Omega}\\
& + \frac 12 \left\langle \big(2 \gamma \widetilde{\bf E}(\omega, b_{e,3}) -  \partial_t \widetilde {\bf E}(\omega, b_{e,3}) \big) \varrho^2\big (\e( u_e)+ U_{e, {\rm sym}}^1\big) \chi\big._{\Omega_e},  \e(u_e) +  U_{e, {\rm sym}}^1 \right\rangle_{G_s, \Omega}.
\end{aligned}
\end{equation}
Thus we obtain that
\begin{equation*}
\begin{aligned}
\mathcal E(u_e, p_e, \partial_t u_f) \leq \lim\inf\limits_{\ve \to 0} \mathcal E^\ve(u_e^\ve, p_e^\ve, \partial_t u^\ve_f)
\leq \limsup\limits_{\ve \to 0}  \mathcal E^\ve(u_e^\ve, p_e^\ve, \partial_t u^\ve_f) =  \mathcal E(u_e, p_e, \partial_t u_f),
\end{aligned}
\end{equation*}
and, hence the strong stochastic two-scale convergence stated in  Lemma.
\end{proof}

\section{Derivation of macroscopic equations for $b_e$ and $c$.}\label{macro_diffusion}
Using  strong stochastic two-scale convergence of $\e(u_e^\ve)$ and $\partial_t u_f^\ve$ we derive macroscopic equations for concentrations of pectins $b_e$ and calcium $c$. First we shall prove convergence of sequences defined on the boundaries of the random microstructure  $\Gamma^\ve$.

\begin{lemma}\label{converg_bound_Gamma}
Consider the random measure $\mu_\omega$ denoting the surface measure of $\Gamma(\omega)$ and define $d\mu^\ve_\omega(x) = \ve^3\,  d\mu_\omega(x/\ve)$.
\begin{itemize}
\item[(i)]  If $\|b^\ve\|_{L^p(G_{e, T}^\ve)} + \|\nabla b^\ve\|_{L^p(G_{e,T}^\ve)} \leq C$ and $b^\ve \to b$ stochastic two-scale, $b \in L^p(0,T; W^{1,p}(G))$, with $p \in (1, \infty)$,  then  for any $\phi\in C^\infty(0,T;C_0^\infty(\mathbb R^3))$
and any $\psi\in C(\Omega)$ we have
\begin{equation}\label{conver_222}
\lim\limits_{\ve\to 0} \int_{G_T} b^\ve(t,x)\,  \phi(t,x) \psi(\T_{x/\ve}\omega) d\mu^\ve_\omega(x) dt  = \int_{G_T} \int_\Omega  b(t,x) \phi(t,x) \psi(\omega) d{\boldsymbol \mu} (\omega) \, dx dt
\end{equation}
and
\begin{equation}
\begin{aligned}
\int_{G_T} \int_\Omega |b|^p d{\boldsymbol \mu} (\omega) dx dt \leq C  \int_{G_T} \int_\Omega |b|^p d\mathcal P dx dt.
\end{aligned}
\end{equation}
\item[(ii)] If
 $\|b^\ve\|_{L^p(G_{e,T}^\ve)} + \ve \|\nabla b^\ve\|_{L^p(G_{e, T}^\ve)} \leq C$ and $b^\ve \to b$ stochastic two-scale, $b \in L^p(G_T;  W^{1,p}(\Omega, d\mathcal P))$, with $p\in (1,\infty)$,  then convergence \eqref{conver_222} holds, and
\begin{equation}\label{omega_lp}
\int_{G_T} \int_\Omega |b|^p d{\boldsymbol \mu} (\omega) dx dt \leq C.
\end{equation}
\end{itemize}
\end{lemma}

\begin{proof}
For $\mathcal P$-a.a. realisations $\omega \in \Omega$, using the assumptions on the geometry of $G_e^\ve$ and  the trace inequality in each  $G^\sigma_{e,j}=G^\sigma_{f, j}(\omega)\setminus G_{f,j}(\omega)$, see proof of Lemma~\ref{cor_estim_extend} for the definition of $G^\sigma_{f, j}(\omega)$,  applying the scaling $x/\ve$ and summing up over $j$ we obtain
\begin{equation}\label{bound_ineq}
\begin{aligned}
\int_{G_T} |b^\ve|^p d\mu^\ve_\omega(x) dt = \ve \int_{\Gamma_{e,T}^\ve} |b^\ve|^p d\sigma^\ve dt  \leq C_1 \int_{G_{e,T}^\ve} | b^\ve|^p dx dt  + C_2 \ve^p \int_{G_{e,T}^\ve} |\nabla b^\ve|^p dx dt  \leq C.
\end{aligned}
\end{equation}
Moreover, in the case {\it (i)} the limit function $b$ does not depend on $\omega$, its trace on $\Gamma_{e,T}^\ve$ is well defined, and
\begin{equation}\label{bound_ineq_bis1}
\begin{aligned}
\ve^p \int_{G_{e,T}^\ve} |\nabla b^\ve-\nabla b|^p dx dt  \
\mathop{\longrightarrow}\limits_{\ve\to0} \ 0.
\end{aligned}
\end{equation}
Choosing $\phi(x,t)=\phi_1(t)\phi_2(x)$ we conclude that $\hat b^\ve(x)=\int_0^T b^\ve(x,t)\phi_1(t)\,dt$
converges in $L^p(G)$ strongly to  $\hat b(x)=\int_0^T b(x,t)\phi_1(t)\,dt$, and
\begin{equation*}
\int_{G} |\hat b^\ve-\hat b|^p d\mu^\ve_\omega(x)  \leq C_3 \int_{G_{e}^\ve} |\hat b^\ve-\hat b|^p dx  + C_4
\ve^p \int_{G_{e}^\ve} |\nabla \hat b^\ve-\nabla \hat b|^p dx  \
\mathop{\longrightarrow}\limits_{\ve\to0} \ 0.
\end{equation*}
This yields \eqref{conver_222}.

   In the case (ii),
for $b \in L^p(G_T; W^{1,p} (\Omega, d\mathcal P))$  using
the same arguments as in the proof of Lemma \ref{lemma_trace_ineq} one can show that
$b\in L^p(G_T; L^p(\Omega,\boldsymbol{\mu}))$. This yields \eqref{omega_lp}.

To justify \eqref{conver_222} we regularize measures $\mu_\omega$ as follows. Let $k=k(x)$ be a non-negative
symmetric $C_0^\infty(\mathbb R^d)$ function such that $\int_{\mathbb R^d}k(x)\,dx=1$, where here $d=3$. We set
$$
d\mu_{\omega,\delta}(x)=\rho^\delta(\mathcal{T}_x\omega)dx \qquad\hbox{with}\ \
\rho^\delta(\omega)=\delta^{-d}\int_{\mathbb R^d}k\Big(\frac y\delta\Big)d\mu_\omega(y).
$$
It is easy to check that a.s. for any test functions $\phi\in C^\infty(0,T;C_0^\infty(\mathbb R^d))$
and $\psi\in C(\Omega)$ we have
$$
\lim\limits_{\delta\to0}\limsup\limits_{\ve\to0}
\bigg|\int_{G_T} b^\ve(t,x)\,  \phi(t,x) \psi(\T_{x/\ve}\omega) d\mu^\ve_\omega(x) dt
-
\int_{G_T} b^\ve(t,x)\,  \phi(t,x) \psi(\T_{x/\ve}\omega) d\mu^\ve_{\omega,\delta}(x) dt\bigg|=0.
$$
The Palm measure of $d\mu_{\omega,\delta}(x)$ is $d\boldsymbol{\mu}_\delta(\omega)=\rho^\delta(\omega)d\mathcal{P}$.
Since  for each $\delta>0$ the measure $\boldsymbol{\mu}_\delta$ is absolutely continuous with respect to $d\mathcal{P}$
and the density $\rho^\delta$ is bounded, the two-scale limit of $b^\ve$  with respect to $d\mu^\ve_{\omega,\delta}$ is $b$
that is
$$
\lim\limits_{\ve\to0}
\int_{G_T} b^\ve(t,x)\,  \phi(t,x) \psi(\T_{x/\ve}\omega) d\mu^\ve_{\omega,\delta}(x) dt
=\int_{G_T}\int_{\Omega}b(t,x,\omega)\phi(t,x)\psi(\omega)d\boldsymbol{\mu}_\delta(\omega)dxdt.
$$
By the trace theorem  a.s.
$$
\limsup\limits_{\ve\to0}\|b^\ve\|_{L^p(G_T,d\mu_\omega^\ve)}\le C.
$$
Therefore, for a subsequence $b^\ve$ stochastically two-scale converge in $L^p(G_T,d\mu_\omega^\ve)$ to some function
$B\in L^p(G_T;L^p(\Omega,d\boldsymbol{\mu}))$.  As was proved in $\cite{Zhikov_Piatnitski_2006}$, the measures $d\boldsymbol{\mu}_\delta$
converge weakly to the measure $d\boldsymbol{\mu}$.  Using one more time the same arguments as in the proof of Lemma
\ref{lemma_trace_ineq} we obtain
$$
\lim\limits_{\delta\to0}\int_{G_T}\int_\Omega b(x,t)\phi(t,x)\psi(\omega)d\boldsymbol{\mu}_\delta(\omega)dx dt=
\int_{G_T}\int_\Omega b(x,t)\phi(t,x)\psi(\omega)d\boldsymbol{\mu}(\omega) dx dt.
$$
Passing to the limit $\delta\to0$ and combining the above relations, we conclude that
$$
\int_{G_T}\int_\Omega b(x,t)\phi(t,x)\psi(\omega)d\boldsymbol{\mu}(\omega)dxdt
=\int_{G_T}\int_\Omega B(x,t)\phi(t,x)\psi(\omega)d\boldsymbol{\mu}(\omega) dx dt.
$$
In view of arbitrariness of $\phi$ and $\psi$ this implies that $B=b$ in $L^p(G_T;L^p(\Omega,d\boldsymbol{\mu}))$.

\end{proof}

Using the convergence on the oscillating boundary $\Gamma^\ve$ proved in Lemma~\ref{converg_bound_Gamma} we can now derive macroscopic equations for $b_e$ and $c$.
\begin{proof}[Proof of Theorem~\ref{th:reactions}]
We can rewrite  the  microscopic equation for $b_e^\ve$ as
\begin{equation}\label{cd_two_scale_b}
\begin{aligned}
&- \langle   b^\ve_e  \, \chi\big._{G_e^\ve} ,\partial_t  \varphi_1\rangle_{G_T}+\langle D^\ve_b(b^\ve_{e,3}) \nabla b^\ve_e, \nabla\varphi_1\,  \chi\big._{G_e^\ve}  \rangle_{G_{T}}
=\langle   b_{e0}  \, \chi\big._{G_e^\ve} , \varphi_1(0)\rangle_{G}\\
&+  \langle g_b(c^\ve_e,b_e^\ve, \e(u_e^\ve)), \varphi_1\,  \chi\big._{G_e^\ve}  \rangle_{G_{T}}
 +  \ve \langle R(b^\ve_e) , \varphi_1 \rangle_{ \Gamma^\ve_T}
 + \langle  F_b(b^\ve_e),  \varphi_1 \rangle_{(\partial G)_T}
\end{aligned}
\end{equation}
for $\varphi_1 = \phi_1(t,x) + \ve \phi_2(t,x)\phi_3(\T_{x/\ve}\omega)$, where
$\phi_1 \in C^\infty(\overline G_T)$, with $\phi_1(T,x) =0 $ for $x\in \overline G$,  $\phi_2 \in C^\infty_0(G_T)$, and $\phi_3 \in  C^1_\T (\Omega)$.

From the a priori estimates for $b^\ve$ and assumptions on $R$ we have that
$$
\ve \int_0^T \int_{\Gamma^\ve} |R(b^\ve_e)|^2 d\sigma^\ve  dt \leq C,
$$
where the constant $C$ is independent of $\ve$.  Thus considering the stochastic two-scale convergence  we obtain that there exists $\widetilde R \in L^2(G_T\times \Omega, dt \times dx \times d {\boldsymbol \mu} (\omega))$
$$
\lim\limits_{\ve \to 0} \ve \int_0^T \int_{\Gamma^\ve} R(b^\ve_e)\,  \varphi_1\, d\sigma^\ve  dt  = \lim\limits_{\ve \to 0} \int_0^T \int_G   R(b^\ve_e)\,  \varphi_1\, d \mu^\ve_{\widetilde\omega}(x) dt  = \int_0^T \int_G  \int_\Omega \widetilde R \varphi_1\, d{\boldsymbol \mu} (\omega) dx dt,
$$
where $\mu_\omega$ is the random measure of  $\Gamma(\omega)$.
Using the assumptions on the geometry and on the function $R$ together with the strong convergence of $b_e^\ve$ in $L^2(G_T)$ we have that
$$
\ve \int_0^T \int_{\Gamma^\ve} |R(b^\ve_e) - R(b_e)|^2  d\sigma^\ve  dt \leq C \int_{G_{e,T}^\ve} \left[|b^\ve_e - b_e|^2  +
\ve^2 |\nabla(b^\ve_e - b_e)|^2 \right] dx dt  \to 0 \text{ as } \ve \to 0.
$$
Then using the strong convergence of $b_e$, the continuity of $R$ and the convergence result in Lemma~\ref{converg_bound_Gamma}  we obtain that   $\widetilde R = R(b_e)$ $\mathcal P$-a.s. in $G_T\times \Omega$.

 Taking  the stochastic two-scale limit   and using the  strong convergence of $b_e^\ve$ and $c_e^\ve$ and the strong stochastic two-scale convergence of $\e(u_e^\ve)$, shown in Lemma~\ref{lem_strong_two-scale},  we obtain
\begin{equation}\label{macro_b_22}
\begin{aligned}
&-  \langle \vartheta_e  b_e,  \partial_t \phi_1 \rangle_{G_T}+\langle D_b(b_{e,3}) (\nabla b_e + B_e^1) \chi\big._{\Omega_e}, \nabla\phi_1 +\phi_2  \nabla_\omega \phi_3 \rangle_{G_{T}, \Omega}
=   \langle \vartheta_e  b_{e0},   \phi_1(0) \rangle_{G}\\
&+ \langle g_b(c,b_e, \e(u_e)+ U^1_{e, {\rm sym}}) \chi\big._{\Omega_e}, \phi_1 \rangle_{G_{T}, \Omega} + \int_0^T \int_{G} \int_\Omega R(b_e)\, \phi_1 d{\boldsymbol \mu(\omega)} dx dt  + \langle  F_b(b_e),  \phi_1 \rangle_{(\partial G)_T}.
\end{aligned}
\end{equation}

To show the convergence of $g_b(c^\ve_e,b_e^\ve, \e(u_e^\ve))$ we considered an approximation of   $U_{e, {\rm sym}}^1\in L^2(G_T\times \Omega)$ by $U_\delta  \in C(G_T; C_\T (\Omega))$,  such that
$U_\delta \to U_{e, {\rm sym}}^1$ in $L^2(G_T\times \Omega)$ as $\delta \to 0$.   For $\mathcal P$-a.a.\  $\omega \in \Omega$ we define
$U_\delta^{\ve} (t,x)= U_\delta (t,x, \T_{x/\ve} \omega)$.  Using the strong stochastic two-scale convergence of $U_\delta^{\ve}$ to $U_\delta$ we obtain
\begin{equation}\label{estim_delta}
\lim\limits_{\delta \to 0} \lim\limits_{\ve \to 0}\|U_\delta^{\ve} \|_{L^2(G_T)}  =\lim\limits_{\delta \to 0}  \| U_\delta \|_{L^2(G_T \times \Omega)}   = \|U_{e, {\rm sym}}^1 \|_{L^2(G_T \times \Omega)}.
\end{equation}
Then  for $\phi_2 \in C_0([0,T)\times G)$ and $\phi_3 \in C_\T(\Omega)$ we can write
\begin{equation*}
\begin{aligned}
& \langle g_b(c_e^\ve,b_e^\ve, \e(u_e^\ve)), \phi_2(t,x) \phi_3(\T_{x/\ve} \omega) \,  \chi\big._{G_e^\ve} \rangle_{G_T} \\
&=
\langle g_b(c_e^\ve,b_e^\ve, \e(u_e^\ve)) - g_b(c,b_e, \e(u_e)+ U^{\ve}_\delta) , \phi_2  \phi_3(\T_{x/\ve} \omega)\,  \chi\big._{G_e^\ve} \rangle_{G_T}\\
&+ \langle  g_b(c, b_e, \e(u_e)+  U^{\ve}_\delta) , \phi_2 \phi_3(\T_{x/\ve} \omega)\,  \chi\big._{G_e^\ve} \rangle_{G_T}.
\end{aligned}
\end{equation*}
Assumptions on $g_b$, together with \eqref{estim_delta},  the strong convergence of $U_\delta$ to $U_{e, {\rm sym}}^1$, and the strong stochastic two-scale convergence of $\e(u_e^\ve)$, imply
\begin{equation*}
\lim\limits_{\delta \to 0} \lim\limits_{\ve \to 0}
\langle  g_b(c, b_e, \e(u_e)+  U_\delta^{\ve}) , \phi_2\, \phi_3 \, \chi\big._{G_e^\ve} \rangle_{G_T} =
\langle  g_b(c, b_e, \e(u_e)+ U_{e, {\rm sym}}^1) , \phi_2 \phi_3 \,  \chi\big._{\Omega_e} \rangle_{G_T,  \Omega}
\end{equation*}
and
\begin{equation*}
\begin{aligned}
& |\langle g_b(c_e^\ve,b_e^\ve, \e(u_e^\ve)) - g_b(c,b_e, \e(u_e)+  U^{\ve}_\delta) , \phi_2\phi_3 \chi_{G_e^\ve} \rangle_{G_T}|
\\
& \leq C \big[\|c_e^\ve - c\|_{L^2(G_T)} + \|b_e^\ve - b_e\|_{L^2(G_T)}  + \|\e(u_e^\ve) - (\e(u_e) + U^{\ve}_\delta)\|_{L^2(G_{e,T}^\ve)} \big]
\\
& \leq C(\ve) + C \big[ \|\e(u_e^\ve)\, \chi\big._{G_e^\ve}\|^2_{L^2(G_T)} + \| (\e(u_e) + U^{\ve}_\delta) \chi\big._{G_e^\ve} \|^2_{L^2(G_T)}
- 2 \langle \e(u_e^\ve)\, \chi\big._{G_e^\ve},  \e(u_e) + U^{\ve}_\delta \rangle_{G_T} \big]^{1/2} \to 0
\end{aligned}
\end{equation*}
as $\ve \to 0$ and $\delta \to 0$.  Assumptions on $g$ and a priori estimates for  $b^\ve_e$, $c_e^\ve$, and $\e(u_e^\ve)$ ensure that
$$
\|g_b(c_e^\ve,b_e^\ve, \e(u_e^\ve)) \|_{L^2(G_{e,T}^\ve)} \leq C,
$$
where the constant $C$ is independent of $\ve$. Thus
$$
g_b(c_e^\ve,b_e^\ve, \e(u_e^\ve)) \rightharpoonup g_b(c_e,b_e, \e(u_e)+ U_{e, {\rm sym}}^1) \quad \text{ stochastically  two-scale}.
$$

Considering $\phi_1=0$  and using the linearity of the resulted equation  we obtain
$$
B^1_e(t,x,\omega)  = \sum_{j=1}^3 \partial_{x_j} b_e(t,x) w^j_b(\omega)
$$
 and the unit cell problem \eqref{unit_cell_prob_b} for $w_b^j$. Choosing $\phi_2=0$ yields macroscopic equations for $b_e$.

Taking  $\varphi_2(t,x)= \psi_1(t,x) + \ve \psi_2(t,x) \psi_3(\T_{x/\ve}\omega)$  with  $\psi_1 \in C^\infty(\overline G_T)$, $\psi_1(T,x) =0$ for $x\in \overline G$,  $\psi_2 \in C^\infty_0(G_T)$, and  $\psi_3 \in C^1_{\T,  \Gamma} (\Omega)$   as a test function in
\eqref{cd_two} we obtain
\begin{equation*}\label{cd_two_scale_c}
\begin{aligned}
&-  \langle c^\ve_e  \chi\big._{G_e^\ve}, \partial_t  \varphi_2\rangle_{G_T} +\langle D_e(b^\ve_{e,3}) \nabla c^\ve_e, \nabla\varphi_2\,  \chi\big._{G_e^\ve}  \rangle_{G_{T}}
-\langle g_c(c^\ve_e,b_e^\ve, \e(u_e^\ve)), \varphi_2 \chi\big._{G_e^\ve}  \rangle_{G_{T}} \\
&-  \langle  c^\ve_f  \chi\big._{G_f^\ve},  \partial_t  \varphi_2  \rangle_{G_T}+\langle D_f \nabla c^\ve_f, \nabla\varphi_2 \chi\big._{G_f^\ve} \rangle_{G_{T}}
-\langle\mathcal G(\partial_t u^\ve_f) c_f^\ve, \nabla\varphi_2 \chi\big._{G_f^\ve}  \rangle_{G_{T}}-\langle g_f(c^\ve_f), \varphi_2 \chi\big._{G_f^\ve} \rangle_{G_{T}}\\
& =   \langle c^\ve_{e0} \,  \chi\big._{G_e^\ve},   \varphi_2(0) \rangle_{G}  +
 \langle  c^\ve_{f0} \,  \chi\big._{G_f^\ve},   \varphi_2(0)  \rangle_{G}+  \langle F_c(c^\ve_e),  \varphi_2 \rangle_{(\partial G)_T}.
\end{aligned}
\end{equation*}
In the same way as for $g_b$, using  the strong stochastic two-scale convergence of $\e(u_e^\ve)\chi\big._{G^\ve_e} $ and  $\partial_t u^\ve_f\, \chi\big._{G^\ve_f} $, the  strong convergence of $b_e^\ve$ and $c^\ve$, and assumptions on  $g_e$ and $\mathcal G$, we  obtain
\begin{equation*}
\begin{aligned}
 \chi\big._{G_e^\ve} \, g_e(c_e^\ve, b_e^\ve, \e(u_e^\ve)) &\rightharpoonup \chi\big._{\Omega_e} \, g_e(c_e, b_e, \e(u_e) + U_{e, {\rm sym}}^1) \qquad && \text{stochastically  two-scale}, \\\
\chi\big._{G_f^\ve} \,  \mathcal G(\partial_t u_f^\ve ) & \rightharpoonup \chi\big._{\Omega_f} \,  \mathcal G(\partial_t u_f ) && \text{stochastically  two-scale}.
\end{aligned}
\end{equation*}
Thus applying the stochastic  two-scale  and  the strong convergences of $b_e^\ve$ and  $c^\ve$,   together with strong stochastic  two-scale convergence of
$\e(u_e^\ve) \, \chi\big._{G_e^\ve}$ and $\partial_t u_f^\ve \, \chi\big._{G_f^\ve}$, yields
\begin{equation}\label{macro_c_1}
\begin{aligned}
& -  \langle   c,  \partial_t \psi_1  \rangle_{G_T,  \Omega}
 + \langle D(b_{e,3}) (\nabla c + C^1),  \nabla\psi_1 + \psi_2 \nabla_\omega \psi_3 \rangle_{G_T, \Omega}  -\langle\mathcal G(\partial_t u_f) c\, \chi\big._{\Omega_f}, \nabla\psi_1 + \psi_2 \nabla_\omega \psi_3  \rangle_{G_{T}, \Omega} \\
& = \langle g_f(c)\, \chi\big._{\Omega_f}, \psi_1  \rangle_{G_{T}, \Omega}+ \langle g_c(c,b_e, \e(u_e)+ U_{e, {\rm sym}}^1) \chi\big._{\Omega_e}, \psi_1   \rangle_{G_{T}, \Omega}
+  \langle   c_0,   \psi_1(0)  \rangle_{G, \Omega}+ \langle F_c(c),  \psi_1 \rangle_{(\partial G)_T}.
\end{aligned}
\end{equation}
Considering $\psi_1= 0$ yields
\begin{equation*}
\langle D(b_{e,3}) (\nabla c+ C^1) -  \mathcal G(\partial_t u_f) \, c\,  \chi\big._{\Omega_f} ,    \psi_2   \nabla_\omega \psi_3 \rangle_{G_T\times \Omega} =0.
\end{equation*}
From here we obtain that
\begin{equation}
\begin{aligned}
& C^1(t,x,\omega)  = \sum_{j=1}^3 \partial_{x_j} c(t,x) w^j(\omega) + c(t,x) Z(t,x,\omega)\,  \chi\big._{\Omega_f},
\end{aligned}
\end{equation}
where  $w^j\in L^2_{{\rm pot}, \Gamma} (\Omega)$, with $j=1, 2,3$, and $Z\in L^\infty(G_T;  L^2_{\rm pot} (\Omega))$ are solutions of the cell problems
\eqref{unit_cell_prob_c} and \eqref{unit_cell_prob_cf}.
Then considering $\psi_2 =0$ and first $\psi_1\in C_0^1(G_T)$ and then $\psi_1 \in C_0^1(0,T; C^1(\overline G))$,   and using the expression  \eqref{corrector_U}  for the corrector  $U_e^1$  we obtain  the macroscopic equation and the boundary conditions  for $c$ in \eqref{macro_concentration}.
The equations for $b_e$  and $c$ and the fact that $b_e, c \in L^2(0,T; H^1(G))$ imply that $\partial_t b_e, \partial_t c \in L^2(0, T; (H^1(G))^\prime)$. Thus $b_e, c \in C([0,T]; L^2(G))$ and using equations   \eqref{macro_b_22}  and  \eqref{macro_c_1} we obtain that  $b_e$ and $c$ satisfy initial conditions.
\end{proof}

\section{Well-posedness of the macroscopic problem}
In the same way  as in the case of periodic microstructure \cite{Andrey_Mariya},  using fixed point iteration we show existence of an unique solution of  the limit problem.
\begin{lemma}\label{estim_u_ef_macro}
There exists a unique weak solution of the limit problem \eqref{macro_ue}--\eqref{macro_two-scale_uf}, \eqref{macro_concentration}.
\end{lemma}
\begin{proof}
First we show estimates for  two iterations  $(u_{e}^{j-1},  \partial_t p_{e}^{j-1}, \partial_t u_{f}^{j-1})$,   $(b_e^{j-1}, c^{j-1})$ and  $(u_{e}^j,  \partial_t p_{e}^j, \partial_t u_{f}^j)$,   $(b_e^j, c^j)$ for  limit problem \eqref{macro_ue}--\eqref{macro_two-scale_uf},~\eqref{macro_concentration}.

We begin with the equations  for fluid flow velocity $\partial_t u_f$ and for elastic displacement  $u_e$.
Taking $\partial_t \widetilde u_{f}  -  \partial_t \widetilde u_{e}$ as a test function in the equation for the difference  $\partial_t \widetilde u_{f}^j$ and $\partial_t \widetilde u_{e}$ as a test function in the equations for the difference  $ \widetilde u_{e}^j$  we obtain
\begin{equation}\label{eq_uniq_u}
\hspace{-0.1 cm }  \begin{aligned}
&\; \; \; \rho_e\|\partial_t  \widetilde u_{e}^j(s)\|^2_{L^2(G)}
+   \big\langle {\bf E}^{\rm hom}(b_{e,3}^{j-1}) {\bf e}(\widetilde u_{e}^j(s)), {\bf e}(\widetilde u_{e}^j(s))\big\rangle_G -  \big\langle \partial_t {\bf E}^{\rm hom}(b_{e,3}^{j-1}) {\bf e}(\widetilde u_{e}^j),{\bf e}(\widetilde u_{e}^j)\big\rangle_{G_s}
\\ &
+2 \big \langle ({\bf E}^{\rm hom}(b_{e,3}^{j}) - {\bf  E}^{\rm hom}(b_{e,3}^{j-1})) \, {\bf e}(u_{e}^j(s)),   \e(\widetilde u_{e}^j(s))\big\rangle_G \\
& - 2 \big\langle \partial_t  ({\bf E}^{\rm hom}(b_{e,3}^{j}) - {\bf E}^{\rm hom}(b_{e,3}^{j-1})) \, {\bf e}(u_{e}^j) + ({\bf E}^{\rm hom}(b_{e,3}^{j}) - {\bf E}^{\rm hom}(b_{e,3}^{j-1}))\,  \partial_t {\bf e}(u_{e}^j) ,  {\bf e}(\widetilde u_{e}^j)\big \rangle_{G_s}
\\
& +  \rho_f \|\partial_t \widetilde u_{f}^j(s)\, \chi\big._{\Omega_f} \|^2_{L^2(G\times \Omega)}  + 2 \mu \|{\bf e}_\omega(\partial_t \widetilde u_{f}^j)\, \chi\big._{\Omega_f}\|^2_{L^2(G_s\times \Omega)}  +2  \langle \nabla \widetilde p_{e}^j, \partial_t \widetilde u_e^j \chi\big._{\Omega_e}+ \partial_t \widetilde u_{f}^j \, \chi\big._{\Omega_f}  \rangle_{G_s, \Omega}
\\ &
 = 2 \langle \widetilde P_e^{1, j}, \partial_t \widetilde u_{f}^j \, \chi\big._{\Omega_e}-\partial_t \widetilde u_{e}^j \, \chi\big._{\Omega_e} \rangle_{G_s,  \Omega}
+  \rho_f \|\partial_t \widetilde u_{f}^j(0) \, \chi\big._{\Omega_f} \|^2_{L^2(G\times \Omega)} + \rho_e \|\partial_t  \widetilde u_{e}^j(0)\|^2_{L^2(G)}
\\
&
+  \big\langle {\bf E}^{\rm hom}(b_{e,3}^{j-1})\, {\bf e}(\widetilde u_{e}^j(0)), {\bf e}(\widetilde u_{e}^j(0))\big\rangle_G + 2 \big \langle ({\bf E}^{\rm hom}(b_{e,3}^{j}) - {\bf  E}^{\rm hom}(b_{e,3}^{j-1})) \, {\bf e}(u_{e}^j(0)),   \e(\widetilde u_{e}^j(0))\big\rangle_G ,
\end{aligned}
\end{equation}
where
  $\widetilde u_e^j = u_e^j - u_e^{j-1}$, $\widetilde p_e^j = p_e^j - p_e^{j-1}$, $\widetilde u_f^j = u_f^j - u_f^{j-1}$,     and  $\widetilde P_e^{1,j} = P^{1,j}_{e} - P^{1,j-1}_{e}$.  The  equation \eqref{macro_ue} for $p_e^j$ and $p_e^{j-1}$ yields
\begin{equation} \label{eq_uniq_p}
\begin{aligned}
\rho_p \|\widetilde p_{e}^j(s)\|^2_{L^2(G)} + 2 \langle K_{p}^{\rm hom} \nabla \widetilde p_{e}^j, \nabla \widetilde  p_{e}^j  \rangle_{G_s}
=  2 \langle K_{u}\,  \partial_t \widetilde  u_e^j  + Q(\partial_t u_f^j)- Q(\partial_t u_f^{j-1}), \nabla \widetilde   p_{e}^j \rangle_{G_s}  \\ + \rho_p \|\widetilde p_{e}^j(0)\|^2_{L^2(G)} .
\end{aligned}
\end{equation}
Due to the  assumptions in {\bf A1} on ${\bf E}$, the definition of the macroscopic elasticity tensor ${\bf E}^{\rm hom}$ and the properties of a solution $W^{kl}_e$, with $k,l=1,2,3$,  of the corresponding cell problems in \eqref{unit_1}, we have
$$
\|{\bf E}^{\rm hom}(b_{e,3}^j) - {\bf E}^{\rm hom}(b_{e,3}^{j-1})\|_{L^\infty(G_s)}+ \|\partial_t  ({\bf E}^{\rm hom}(b_{e,3}^j) - {\bf E}^{\rm hom}(b_{e,3}^{j-1})) \|_{L^\infty(G_s)} \leq  C \|\widetilde b_{e}^j \|_{L^\infty(0,s; L^\infty(G))}
$$
for  $s \in (0,T]$, where  $\widetilde b_e^j = b_e^j - b_e^{j-1}$.  The  expression \eqref{expres_Pe1} for  $P_{e}^{1,j}$ and $P_e^{1,j-1}$    and the estimates for the  $H^1$-norm  of the  solutions of the  cell problems  \eqref{unit_1} and \eqref{two-scale_qf} yield
\begin{equation*}
\begin{aligned}
\| \widetilde P^{1,j}_{e} \|_{L^2(G_s\times \Omega)}  \leq  C \left(\|\nabla\widetilde p_{e}^j  \|_{L^2(G_s)} +  \|\partial_t  \widetilde u_{e}^j \|_{L^2(G_s)}
+ \|\partial_t  \widetilde u_{f}^j\,  \chi\big._{\Omega_f} \|_{L^2(G_s\times \Omega)} \right).
\end{aligned}
\end{equation*}
Adding  \eqref{eq_uniq_u} and \eqref{eq_uniq_p}, and applying the  H\"older and Gronwall inequalities yield
\begin{equation}\label{uefp_contact}
\begin{aligned}
  \|\partial_t  \widetilde u_{e}^j\|_{L^\infty(0,s; L^2(G))} + \| {\bf e}(\widetilde u_{e}^j) \|_{L^\infty(0,s; L^2(G))}
 +  \|\widetilde p_{e}^j\|_{L^\infty(0,s; L^2(G))}  + \|\nabla \widetilde p_{e}^j\|_{ L^2(G_s)} \\
+  \|\partial_t  \widetilde u_{f}^j \, \chi\big._{\Omega_f}\|_{L^\infty(0,s; L^2(G\times \Omega))}  +  \|{\bf e}_\omega(\partial_t  \widetilde u_{f}^j) \, \chi\big._{\Omega_f} \|_{L^2(G_s\times \Omega)}
  \leq C  \|\widetilde b_{e}^j \|_{L^\infty(0, s; L^\infty(G))}
\end{aligned}
\end{equation}
for all $s\in (0,T]$ and the constant $C$ does not depend on $s$ and solutions of the macroscopic problem.

In the same way as in  the case of periodic microstructure  \cite{Andrey_Mariya} we obtain the following estimates for
 $\widetilde b_e^j$ and $\widetilde c^j$:
  \begin{equation}\label{bc_contact}
  \|\widetilde  b_{e}^j\|_{L^\infty(0,s; L^\infty(G))} + \| \widetilde c^j\|_{L^\infty(0, s; L^2(G))} \leq C_1\big[ \|{\bf e}(\widetilde u_{e}^{j-1})\|_{L^{1 + \frac 1 \sigma}(0,s; L^2(G))} +\| \partial_t \widetilde u_{f}^{j-1}\, \chi\big._{\Omega_f} \|_{L^2(G_s\times \Omega)}  \big], \\
\end{equation}
for  $s\in (0,T]$  and   any  $0 < \sigma< 1/9$, the constant $C$ being independent of $s$ and solutions of the problem,
and
$$
  \| b_{e}^j \|_{L^\infty(0,T; L^\infty(G))}  +\| c^j \|_{L^\infty(0,T; L^\infty(G))}  + \| b_{e}^{j-1} \|_{L^\infty(0,T; L^\infty(G))}  +\| c^{j-1} \|_{L^\infty(0,T; L^\infty(G))}  \leq C.
  $$
Then combining \eqref{uefp_contact} and  \eqref{bc_contact} and  applying a fixed point argument we obtain existence of a uniques solution of the  coupled macroscopic problem \eqref{macro_ue}--\eqref{macro_two-scale_uf}, \eqref{macro_concentration}.
\end{proof}

\section*{Acknowledgements}   This research was supported by a Northern Research Partnership early career research exchange grant. The research of MP was also supported by the EPSRC First grant EP/K036521/1 ''Multiscale modelling and analysis of mechanical properties of plant cells and tissues.''

\end{document}